\documentclass[11pt,reqno]{amsart}%
\usepackage{graphicx}
\usepackage{amsmath}
\usepackage{mathtools, amsfonts, amssymb, amsthm}
\usepackage{mathrsfs}
\usepackage{enumerate}
\usepackage{algorithmic}
\usepackage{boxedminipage}
\usepackage{color}
\usepackage{cite}
\usepackage{url}
\usepackage[margin=1in]{geometry}%
\usepackage{tikz}
\usepackage{multirow}
\usepackage{array}

\usetikzlibrary{arrows,shapes,chains}
\providecommand{\U}[1]{\protect\rule{.1in}{.1in}}

\newtheorem{theorem}{Theorem}[section]
\newtheorem{proposition}[theorem]{Proposition}
\newtheorem{lemma}[theorem]{Lemma}

\theoremstyle{remark}
\newtheorem{example}[theorem]{Example}

\makeatletter
\newcommand{\thickhline}{%
    \noalign {\ifnum 0=`}\fi \hrule height 1pt
    \futurelet \reserved@a \@xhline
}
\newcolumntype{"}{@{\hskip\tabcolsep\vrule width 1pt\hskip\tabcolsep}}
\makeatother

\begin{document}
\title{Best Nonnegative Rank-One Approximations of Tensors}

\author{Shenglong Hu}
\address{Department of Mathematics, School of Science, Hangzhou Dianzi University, Hangzhou 310018, China.}
\email{shenglonghu@hdu.edu.cn}

\author{Defeng Sun}
\address{Department of Applied Mathematics, The Hong Kong Polytechnic University, Hung Hom, Hong Kong.}
\email{defeng.sun@polyu.edu.hk}

\author{Kim-Chuan Toh}
\address{Department of Mathematics, and Institute of Operations Research and Analytics, National University of Singapore, 10 Lower Kent Ridge Road, Singapore.}
\email{mattohkc@nus.edu.sg}

\begin{abstract}
In this paper, we study the polynomial optimization problem of multi-forms over the intersection of the multi-spheres and the nonnegative orthants.
This class of problems is NP-hard in general, and includes the problem of finding the best nonnegative rank-one approximation of a given tensor. A Positivstellensatz is given for this class of polynomial optimization problems, based on which a globally convergent hierarchy of doubly nonnegative (DNN) relaxations is proposed. A (zero-th order) DNN relaxation method is applied to solve these problems, resulting in linear matrix optimization problems under both the positive semidefinite and nonnegative conic constraints. A worst case approximation bound is given for this relaxation method. Then, the recent solver SDPNAL+ is adopted to solve this class of matrix optimization problems. Typically, the DNN relaxations are tight, and hence the best nonnegative rank-one approximation of a tensor can be revealed 
frequently.
 Extensive numerical experiments show that this approach is quite promising.
\end{abstract}
\keywords{Tensor, nonnegative rank-1 approximation, polynomial, multi-forms, doubly nonnegative semidefinite program, doubly nonnegative relaxation method}
\subjclass[2010]{15A18; 15A42; 15A69; 90C22}
\maketitle

\section{Introduction}\label{sec:introduction}
Nonnegative factorizations of data observations are prevalent in data analysis, which have been popularized to an unprecedented level since the works 
of Paatero and Tapper \cite{PT-94}, and Lee and Seung \cite{LS-99}. In many applications, data are naturally represented by the third order or higher order tensors (a.k.a. hypermatrices). For example, a color image is stored 
 digitally
as a third order nonnegative tensor comprised of three nonnegative matrices, representing the red, green and blue pixels, and therefore a set of such images or a video is actually a fourth order nonnegative tensor. In the literature, however, these fourth order tensors are
typically flattened into matrices 
before data analysis is
applied \cite{PT-94,LS-99,Cic-Zdu-Pha-Ama:non,HPS-05}. As we can see,
the intrinsic structures of an image or a video are destroyed
after the flattening. Therefore, direct treatments of tensors are necessary, and
correspondingly
nonnegative factorizations of higher dimensional data are needed. As a result, tensor counterparts of the nonnegative matrix factorizations have 
become a new frontier
in this area \cite{SH-05,WW-01,Roy-Thi-Com:non,Paa:non,Lev:und,Kol-Bad:sur,All-Rho-Stu-Zwi:non,De-De-Van:sin}. As expected, nonnegative tensor factorizations have their own advantages over the traditional nonnegative matrix factorizations, see for examples \cite{HPS-05,SL-01} and references therein.

Nonnegative tensor factorizations have 
found
diverse applications, such as latent class models in statistics, spectroscopy, sparse image coding in computer vision, sound source separation, and pattern recognition, etc., see \cite{SH-05,HPS-05,FCC-05,Cic-Zdu-Pha-Ama:non} and references therein. Several methods
have been proposed for computing nonnegative tensor factorizations, see \cite{Bro-De:non,Coh-Rot:non,Fri-Hat:non,Lim-Com:non,Paa:non,Sha-Zas-Haz:sym,Zho-Cic-Zha-Xie:tuc,WW-01,Law-Han:lea} and references therein. Due to errors in measurements of the data collected or
simply because of inattainability, the problem of approximating a given tensor by a nonnegative tensor factorization occurs more often in practice than the problem
of finding the exact factorization of a given tensor.
Existence and uniqueness of nonnegative tensor factorizations are well studied in \cite{Lim-Com:non,Qi-Com-Lim:non}.
For a given tensor, a classical method to compute a nonnegative tensor factorization/approximation is by \textit{multiple best nonnegative rank-one approximations}, proposed by Shashua and Hazan \cite{SH-05}. The principle is alternatively splitting/approximating the given tensor by several (nonnegative) ones and approximating each (nonnegative) tensor by a best nonnegative rank-one tensor.
Therefore, in this framework, \textit{finding the best nonnegative rank-one approximation of a given tensor} is of crucial importance in nonnegative tensor factorizations/approximations.
This problem is also the cornerstone of the heuristic methods based on greedy rank-one downdating for nonnegative factorizations \cite{AG-06,BGV-08,BIB-03,G-06}. 

This article will focus on the problem of computing the best nonnegative rank-one approximation of a given tensor from the perspective of mathematical optimization. The problem will be formulated as a polynomial minimization problem over the intersection of the multi-sphere and the nonnegative orthant. With this formulation, the study can also be applied to the problem of testing the copositivity for a homogeneous polynomial, which is important in the completely positive programming \cite{PVZ-14}.

A negative
aspect from the computational complexity point of view is that the
 problem under consideration is NP-hard in general \cite{MK-87,DG-14,Has:ran,Vav:non}. Thus, no algorithm with polynomial complexity exists unless P=NP. Consequently, in practical applications, approximation or relaxation methods are
 employed
 to solve this problem. In this article, instead of adopting the traditional sums of squares (SOS) relaxation methods for a polynomial optimization problem
 (cf.\ \cite{L-01,PS-03,N-12,N-14,NDS-06}), we will introduction a \textit{doubly nonnegative (DNN) relaxation method} to solve this problem. DNN relaxation methods will provide tighter approximation results, since the cone of SOS polynomials in strictly contained in the cone of polynomials which can be written as sums of SOS polynomials and polynomials with nonnegative coefficients.
While the standard SOS relaxations of a polynomial optimization problem will
give rise to standard SDP problems with variables in the SDP cones, the DNN relaxation method
will give rise to DNNSDP problems whose variables are constrained to be in the
SDP cones and the cones of nonnegative matrices, in addition to linear equality
constraints. It has been well recognized that solving the DNNSDP problems by primal-dual
interior-point methods as implemented in popular solvers such as Mosek, SDPT3 \cite{SDPT3}, or SeDuMi \cite{SDM}
is computationally much more challenging than solving the standard SDP counterparts.
Fortunately, with the recent advances on augmented Lagrangian based methods
for solving SDP problems with bound constraints \cite{STY-15,YST-15},
we have reached a stage where solving the DNNSDP problems are computationally
not much more expensive than the standard SDP counterparts.
In this paper, we will employ the Newton-CG augmented Lagrangian method
implemented in the solver SDPNAL+ \cite{YST-15} to solve the DNNSDP problems
arising from best nonnegative rank-one tensor approximation problems.
Extensive numerical computations will show that our new approach is quite promising.

The remaining parts of this article are organized as follows. Preliminaries will be given in Section~\ref{sec:preliminary}, in which nonnegative tensor approximations and in particular the best nonnegative rank-one approximation problems will be presented in Section~\ref{sec:non-app}, and the problem of testing the copositivity of a tensor will be given in Section~\ref{sec:cop}. Both the problems in Section~\ref{sec:preliminary} will be formulated as minimizing a multi-form over the intersection of the multi-sphere and the nonnegative orthants in Section~\ref{sec:homo}. In ensuing section, basic properties of this polynomial optimization problem will be investigated, including a Positivstellensatz for this problem (cf.\ Section~\ref{sec:positive}), the DNN relaxation (cf.\ Section~\ref{sec:sos}), a worst case approximation bound (cf.\ Section~\ref{sec:bound}), and the
extraction of a nonnegative rank-one tensor from a solution of the DNN problem (cf.\ Sections~\ref{sec:solution-even} and \ref{sec:solution-odd}). Numerical computations will be presented in Section~\ref{sec:numerical}, in which extensive examples on best nonnegative rank-one approximations and examples on testing the copositivity of a tensor will be given. Some conclusions will be given in the last section.

\section{Preliminaries}\label{sec:preliminary}
In this article, tensors will be considered in the most general setting.
Given positive integers $n_1,\dots,n_r$, a tensor $\mathcal A\in\mathbb R^{n_1}\otimes\dots\otimes\mathbb R^{n_r}$ is a collection of $n_1\cdots n_r$ scalars $a_{i_1\dots i_r}$, termed the entries of $\mathcal A$, for all $i_j\in\{1,\dots,n_j\}$ and $j\in\{1,\dots,r\}$. If $n_1=\dots=n_r=n$, $\mathbb R^{n_1}\otimes\dots\otimes\mathbb R^{n_r}$ is abbreviated as $\otimes^r\mathbb R^{n}$.
Given positive integers $p$, $\alpha_1,\dots,\alpha_p$, $n_1,\dots,n_p$, we denote by $\operatorname{Sym}(\otimes^{\alpha_i}\mathbb R^{n_i})$ the symmetric  subspace of the tensor space $\otimes^{\alpha_i}\mathbb R^{n_i}$, consisting of real symmetric tensors with order $\alpha_i$ and dimension $n_i$, and $\operatorname{Sym}(\otimes^{\alpha_1}\mathbb R^{n_1})\otimes\dots\otimes\operatorname{Sym}(\otimes^{\alpha_p}\mathbb R^{n_p})$ the tensor space with $p$ symmetric factors. Note that when $p=1$, the tensor space is the usual space of symmetric tensors; and when $\alpha_1=\dots=\alpha_p=1$, the tensor space is the usual space of non-symmetric tensors.
A tensor $\mathcal A\in\operatorname{Sym}(\otimes^{\alpha_1}\mathbb R^{n_1})\otimes\dots\otimes\operatorname{Sym}(\otimes^{\alpha_p}\mathbb R^{n_p})$ is usually referred
to as a \textit{partially symmetric} tensor, which appears in many applications. A symmetric rank-one tensor in $\operatorname{Sym}(\otimes^{\alpha_i}\mathbb R^{n_i})$ is an element $\big(\mathbf x^{(i)}\big)^{\otimes \alpha_i}$ for some vector $\mathbf x^{(i)}\in\mathbb R^{n_i}$, where $\big(\mathbf x^{(i)}\big)^{\otimes \alpha_i}$ is a short hand for
\[
\underbrace{\mathbf x^{(i)}\otimes\dots\otimes\mathbf x^{(i)}}_{\alpha_i\ \text{copies}}.
\]
Therefore, a rank-one tensor in $\operatorname{Sym}(\otimes^{\alpha_1}\mathbb R^{n_1})\otimes\dots\otimes\operatorname{Sym}(\otimes^{\alpha_p}\mathbb R^{n_p})$ is of the form $\mathbf x^{\alpha}:=\big(\mathbf x^{(1)}\big)^{\otimes\alpha_1}\otimes\dots\otimes\big(\mathbf x^{(p)}\big)^{\otimes\alpha_p}$ for some vectors $\mathbf x^{(i)}\in\mathbb R^{n_i}$, $i=1,\ldots,p.$

As an Euclidean space, the inner product $\langle\mathcal A,\mathcal B\rangle$ of two tensors $\mathcal A, \mathcal B\in\mathbb R^{n_1}\otimes\dots\otimes\mathbb R^{n_r}$ is defined as
\[
\langle\mathcal A,\mathcal B\rangle:=\sum_{i_1=1}^{n_1}\dots\sum_{i_r=1}^{n_r}a_{i_1\dots i_r}b_{i_1\dots i_r}.
\]
The Hilbert-Schmidt norm $\|\mathcal A\|$ is then defined as
\[
\|\mathcal A\|:=\sqrt{\langle\mathcal A,\mathcal A\rangle}.
\]
We refer the readers to \cite{Lim:hyp} and references therein for basic notions on tensors.
\subsection{Nonnegative tensor approximation}\label{sec:non-app}
In the context of computer vision, chemometrics, statistics, and spectral intensity, the multi-way (tensor) data often cannot take negative values. Therefore, one expects to approximate as much as possible the observed data with a summation of rank-one nonnegative tensors
\begin{equation}\label{nonnegative-factorization}
\mathcal A\approx \sum_{i=1}^r\lambda_i\mathbf x^{\otimes\alpha}_i\ \lambda_i\geq 0,\ \mathbf x_i\geq \mathbf 0
\end{equation}
for some nonnegative integer $r$,
with
\[
\mathbf x_i:=(\mathbf x^{(1)}_i,\dots,\mathbf x^{(p)}_i)\in\mathbb R^{n_1}\times\dots\times\mathbb R^{n_p},
\]
and
\[
\mathbf x^{\otimes\alpha}_i:=\big(\mathbf x^{(1)}_i\big)^{\otimes\alpha_1}\otimes\dots\otimes\big(\mathbf x^{(p)}_i\big)^{\otimes\alpha_p}.
\]
For a given continuous distance measure $\phi$ over the tensor space, we can formulate problem \eqref{nonnegative-factorization} as
\begin{equation}\label{nonnegative-optimization}
\min\bigg\{\phi(\mathcal A,\sum_{i=1}^r\lambda_i\mathbf x^{\otimes\alpha}_i) : \lambda_i\geq 0,\ \mathbf x_i\geq \mathbf 0\bigg\}.
\end{equation}
In most cases, $\phi$ is chosen as the Hilbert-Schmidt norm  distance, i.e., $\phi(\mathcal A,\mathcal B):=\|\mathcal A-\mathcal B\|$.
Problem \eqref{nonnegative-optimization} is well-defined for each $r\in\mathbb N$, while NP-hard in most cases. Moreover, numerical difficulty arises for problem \eqref{nonnegative-factorization} when we do not know prior $r$, and even if we are luck enough to know the exact $r$, it is still very difficulty to solve \eqref{nonnegative-optimization}.
Thus, one procedure to solve \eqref{nonnegative-factorization} is by multiple best nonnegative rank-one approximations and another is by successive best nonnegative rank-one approximations.

Therefore, we focus on the problem \eqref{nonnegative-optimization} with fixed $r=1$ in this article, i.e., \textit{the best nonnegative rank-one approximation} of the tensor $\mathcal A$. We will see that this problem is already hard,  both theoretically and numerically. The computational complexity is NP-hard in general.

With the common choice of $\phi$ as the Hilbert-Schmidt norm distance, problem \eqref{nonnegative-optimization} becomes
\begin{equation}\label{nonnegative-rank-one}
\begin{array}{rl} \min_{\lambda,\mathbf x}&\|\mathcal A-\lambda\mathbf x^{\otimes \alpha}\|^2\\[5pt]
\text{s.t.}&\lambda\geq 0,\ \langle\mathbf x^{(i)},\mathbf x^{(i)}\rangle=1, \ \mathbf x^{(i)}\geq \mathbf 0,\ \text{for all }i=1,\dots,p, \end{array}
\end{equation}
where $\mathbf x:=(\mathbf x^{(1)},\dots,\mathbf x^{(p)})$.
It is easy to see that \eqref{nonnegative-rank-one} always has an optimal solution $(\lambda,\mathbf x)$ with
\begin{equation}\label{lambda-value}
\lambda := \begin{cases}\langle\mathcal A,\mathbf x^{\otimes \alpha}\rangle,& \text{whenever }\langle\mathcal A,\mathbf x^{\otimes \alpha}\rangle>0,\\0,&\text{otherwise},\end{cases}
\end{equation}
and in both cases
\[
\|\mathcal A-\lambda\mathbf x^{\otimes \alpha}\|^2=\|\mathcal A\|^2 - \lambda^2.
\]
Therefore, \eqref{nonnegative-rank-one} is equivalent to
\begin{equation}\label{minimal-reformulation}
\begin{array}{rl} \min&\langle-\mathcal A,\mathbf x^{\otimes \alpha}\rangle
\\[5pt]
\text{s.t.}& \langle\mathbf x^{(i)},\mathbf x^{(i)}\rangle=1, \ \mathbf x^{(i)}\geq \mathbf 0,\ \text{for all }i=1,\dots,p \end{array}
\end{equation}
in the sense that
\begin{enumerate}
\item if the optimal value of \eqref{minimal-reformulation} is nonnegative, then the zero tensor is the best nonnegative rank-one approximation of $\mathcal A$,
\item if the optimal value $\lambda$ of \eqref{minimal-reformulation} is negative with optimal solution $\mathbf x_*$, then $-\lambda\mathbf x^{\otimes\alpha}_*$ is the best nonnegative rank-one approximation of $\mathcal A$.
\end{enumerate}
\subsection{Copositivitiy of tensors}\label{sec:cop}
A  tensor $\mathcal A\in\operatorname{Sym}(\otimes^{\alpha_1}\mathbb R^{n_1})\otimes\dots\otimes\operatorname{Sym}(\otimes^{\alpha_p}\mathbb R^{n_p})$ is called \textit{copositive}, if
\[
\langle\mathcal A,\mathbf x^{\otimes \alpha}\rangle\geq 0\ \text{for all }\mathbf x\in\mathbb R^{n_1}_+\times\dots\times\mathbb R^{n_p}_+.
\]
The copositivity of a tensor is a generalized notion of both the nonnegativity of a matrix and the copositivity of a symmetric matrix. When $p=1$ and $\alpha_1=2$, it reduces to the copositivity of a symmetric matrix; and when $\alpha_1=\dots=\alpha_p=1$, it reduces to the nonnegativity of a tensor. The problem of deciding the copositivity of a tensor is therefore co-NP-hard \cite{DG-14,MK-87}. When $p=1$, discussions on copositive tensors can be found in \cite{Qi:cop,NYZ-18} and references therein.

Testing the copositivity of a tensor can also be formulated as a polynomial optimization problem as in \eqref{minimal-reformulation}.
Indeed, a tensor $\mathcal A$ is copositive if and only if the optimal value of
\begin{equation}\label{minimal-copositive}
\begin{array}{rl} \min&\langle\mathcal A,\mathbf x^{\otimes \alpha}\rangle\\ \text{s.t.}& \langle\mathbf x^{(i)},\mathbf x^{(i)}\rangle=1, \ \mathbf x^{(i)}\geq \mathbf 0,\ \text{for all }i=1,\dots,p \end{array}
\end{equation}
is nonnegative.

\section{Homogeneous Polynomials}\label{sec:homo}
Since both the problem of finding the best nonnegative rank-one approximation of a tensor and the copositivity certification of a tensor can be equivalently reformulated as \eqref{minimal-reformulation} (or \eqref{minimal-copositive}), we
focus on this polynomial optimization problem in this section.

Let $\mathbf x:=(\mathbf x^{(1)},\dots,\mathbf x^{(p)})\in\mathbb R^{n_1}\times\dots\times\mathbb R^{n_p}$. A polynomial $f(\mathbf x)$ is \textit{multi-homogeneous} or \textit{multi-form}, if each monomial of $f$ has the same degree with respect to each group variables $\mathbf x^{(i)}$ for all $i\in\{1,\dots,p\}$.
We consider the following optimization problem
\begin{equation}\label{nonnegative-problem}
\begin{array}{rrl}f_{\min}:=& \min&f(\mathbf x^{(1)},\dots,\mathbf x^{(p)})\\ &\text{s.t.}& \|\mathbf x^{(i)}\|=1, \ \mathbf x^{(i)}\geq \mathbf 0, \ \mathbf x^{(i)}\in\mathbb R^{n_i},\ i=1,\dots,p,\end{array}
\end{equation}
where $f(\mathbf x^{(1)},\dots,\mathbf x^{(p)})\in\mathbb R[\mathbf x]$ is a multi-form of even degree $d_i=2\tau_i$ for some $\tau_i\geq 0$ with respect to each $\mathbf x^{(i)}$  for all $i\in\{1,\dots,p\}$. 
Problem~\eqref{nonnegative-problem} covers all instances of minimizing a multi-form over the intersection of the multi-sphere and the nonnegative orthants, since the cases with odd $d_i$'s can be equivalently formulated into \eqref{nonnegative-problem} as in Section~\ref{sec:odd}. Polynomial optimization over the multi-sphere is one research direction in recent years, see \cite{LNQY-09,N-12,Nie-Wan:non,NZ-16} and references therein.
Moreover, in \cite{LNQY-09} a biquadratic optimization over the joint sphere (multi-sphere with $p=2$) with one group variables being nonnegative is discussed as well.

For easy references, in the following, we will denote the $n-1$-dimensional sphere in $\mathbb R^n$ as $\mathbb S^{n-1}$, i.e., $\mathbb S^{n-1}:=\{\mathbf x\in\mathbb R^n\colon \mathbf x^\mathsf{T}\mathbf x=1\}$. The nonnegative part of the $n-1$-dimensional sphere is denoted by $\mathbb S_+^{n-1}$, i.e., $\mathbb S_+^{n-1}:=\{\mathbf x\in\mathbb R_+^n\colon \mathbf x^\mathsf{T}\mathbf x=1\}$. Thus, the feasible set of \eqref{nonnegative-problem} can be called as the \textit{nonnegative multi-sphere}.

\subsection{Odd order case}\label{sec:odd}
If $f(\mathbf x^{(1)},\dots,\mathbf x^{(p)})$ is of odd degree $d>0$ for $\mathbf x^{(1)}$ (without loss of generality), then we introduce a variable $t$ and let
\[
\tilde f(\tilde{\mathbf x}^{(1)},\mathbf x^{(2)},\dots,\mathbf x^{(p)}):=tf(\mathbf x^{(1)},\dots,\mathbf x^{(p)})
\]
with $\tilde{\mathbf x}^{(1)}=((\mathbf x^{(1)})^\mathsf{T},t)^\mathsf{T}$.
It can be shown that
\[
 f_{\min}=\sqrt{\frac{(d+1)^{d+1}}{d^d}}\tilde f_{\min},
\]
since
\[
\max\{t\alpha^d : \alpha^2+t^2=1\}=\sqrt{\frac{d^d}{(d+1)^{d+1}}}
\]
with a positive optimal $t$.

If the degree of $f$ for $\mathbf x^{(1)}$ is one, we can construct
\[
g(\mathbf x^{(2)},\dots,\mathbf x^{(p)}):=\sum_{j=1}^{n_1}\big(f(\mathbf e^{(1)}_j,\mathbf x^{(2)},\dots,\mathbf x^{(p)})\big)^2,
\]
where $\mathbf e^{(1)}_j\in\mathbb R^{n_1}$ is the $j$th standard basis vector. In some cases, \eqref{nonnegative-problem} can be solved via minimizing $g$ over the nonnegative multi-sphere, i.e.,
\begin{equation}\label{eq:nonnegative-linear}
\min\{g(\mathbf x^{(2)},\dots,\mathbf x^{(p)})\colon \|\mathbf x^{(i)}\|=1, \ \mathbf x^{(i)}\geq \mathbf 0, \ \mathbf x^{(i)}\in\mathbb R^{n_i},\ 2=1,\dots,p\}.
\end{equation}
Actually, if $(\mathbf x^{(2)},\dots,\mathbf x^{(p)})$ is an optimal solution of \eqref{eq:nonnegative-linear} with positive optimal value and
$f(\mathbf e^{(1)}_j,\mathbf x^{(2)},\dots,\mathbf x^{(p)})$'s are all nonpositive, then we can construct a solution for \eqref{nonnegative-problem} from a solution for \eqref{eq:nonnegative-linear}. Indeed, one optimal solution of \eqref{nonnegative-problem} is given by
\[
(\mathbf x^{(1)}:=-\frac{\mathbf u_{-}}{\|\mathbf u\|},\mathbf x^{(2)},\dots,\mathbf x^{(p)})
\]
with
\[
\mathbf u:=(f(\mathbf e^{(1)}_1,\mathbf x^{(2)},\dots,\mathbf x^{(p)}),\dots,f(\mathbf e^{(1)}_{n_1},\mathbf x^{(2)},\dots,\mathbf x^{(p)}))^\mathsf{T}
\]
and $(\mathbf u_-)_i:=\min\{0,u_i\}$. This is based on the fact that
\[
\min\{\mathbf x^\mathsf{T}\mathbf y : \|\mathbf y\|=1, \ \mathbf y\geq \mathbf 0\}=-\|\mathbf x_-\|
\]
with the optimizer $\mathbf y^*:=-\frac{\mathbf x_-}{\|\mathbf x_-\|}$ when $\mathbf x_-\neq\mathbf 0$. Note that the number of variables is reduced from \eqref{nonnegative-problem} to \eqref{eq:nonnegative-linear}.

Before proceeding to the computation of \eqref{nonnegative-problem}, we state the computational complexity of it.

\subsection{NP-hardness}\label{sec:np-hard}
\begin{proposition}\label{prop:np-hard}
Problem~\eqref{nonnegative-problem} is NP-hard in general.
\end{proposition}

\begin{proof}
We will construct a subclass of \eqref{nonnegative-problem}, which is NP-hard.
Let $G=(V,E)$ be a simple graph with the set of vertices being $V=\{1,\dots,n\}$ and the set of edges being $E$. Let $\Delta_n\subset\mathbb R^n_+$ be the standard simplex. Then
\[
1-\frac{1}{\alpha(G)}=2\max_{\mathbf x\in \Delta_n}\sum_{(i,j)\in E}x_ix_j
\]
by the famous Motzkin-Straus theorem \cite{MS-65}, where $\alpha(G)$ is the stability number of $G$.  It is well known that computing $\alpha(G)$ is an NP-hard problem \cite{MK-87,GJ-80}.  On the other hand, we have that
\[
\max_{\mathbf x\in \Delta_n}\sum_{(i,j)\in E}x_ix_j=\max_{\|\mathbf y\|=1}\sum_{(i,j)\in E}y_i^2y_j^2=\max_{\|\mathbf y\|=1,\ \mathbf y\geq\mathbf 0}\sum_{(i,j)\in E}y_i^2y_j^2,
\]
where the second equality follows from the fact that in the objective function only squared $y_i^2$'s are involved. Immediately, the last optimization problem is of the
form given in \eqref{nonnegative-problem}. The required result then follows.
\end{proof}

A standard SOS relaxation can be applied to the polynomial optimization problem \eqref{nonnegative-problem}, see \cite{L-01}.
However, in order to reduce the size of the resulting SDP, we would like to combine the spherical constraints as follows.

The homogeneity property implies that \eqref{nonnegative-problem} is equivalent to
\begin{equation}\label{nonnegative-problem-homo}
\begin{array}{rrl}f_{\min}:=& \min&f(\mathbf x^{(1)},\dots,\mathbf x^{(p)})\\ &\text{s.t.}& \prod_{i=1}^p\|\mathbf x^{(i)}\|^{d_i}=1, \\ & & \mathbf x^{(i)}\geq \mathbf 0, \ \mathbf x^{(i)}\in\mathbb R^{n_i},\ i=1,\dots,p\end{array}
\end{equation}
in the sense that they have the same optimal objective value and we can get an optimal solution for one from the other.

\subsection{A Positivstellensatz}\label{sec:positive}
Testing the nonnegativity of a polynomial over a (compact) semialgebraic set is a very difficult problem \cite{BCR-98}. Thus, certifications of nonnegativity of a polynomial are foundations for polynomial optimization \cite{L-01}. In the literature, such certifications are called \textit{Positivstellensatz}. Of crucial importance are Putinar's Positivstellensatz \cite{P-93}, P\'olya's theorem \cite{P-28} and Reznick's theorem \cite{R-00}.

While Putinar's result is more general, and the theorems of
P\'olya and Reznick  are applicable only to homogeneous polynomials over the simplices and spheres respectively, the resulting SDP problems obtained from the latter two theorems have sizes that are about half
of those obtained by using Putinar's Positivstellensatz directly. Since the cost of solving SDP problems grow rapidly with the sizes of problems, P\'olya's theorem and Reznick's theorem are more important for homogeneous problems.

In this section, we will derive a Positivstellensatz for the optimization problem \eqref{nonnegative-problem-homo} by taking into account both the homogeneity structures of the objective function and constraints, as well as the nonnegativity constraints.

Let $g(\mathbf x):=\prod_{i=1}^p\|\mathbf x^{(i)}\|^{d_i}$ and $\mathcal F$ be the feasible set of  problem \eqref{nonnegative-problem-homo}.
Suppose that $\gamma:=f_{\min}$ is the optimal value of  \eqref{nonnegative-problem-homo}. It follows that
\[
f(\mathbf x)-\gamma g(\mathbf x)\geq 0\ \text{for all }\mathbf x\in\mathcal F.
\]
We then have
\[
f(\mathbf x)-\gamma g(\mathbf x)\geq 0\ \text{for all }\mathbf x\in\mathbb S_+^{n_1-1}\times\dots\times\mathbb S_+^{n_p-1},
\]
which is equivalent to
\[
f(\mathbf x)-\gamma g(\mathbf x)\geq 0\ \text{for all }\mathbf x\in\Delta_{n_1}\times\dots\times\Delta_{n_p},
\]
where $\Delta_{n_i}$ is the standard simplex in $\mathbb R^{n_i}$, i.e., $\Delta_{n_i}:=\{\mathbf x\in\mathbb R_+^{n_i}\colon\mathbf e^\mathsf{T}\mathbf x=1\}$.
In the following, we will discuss the positivity of a multi-form over the joint simplex.
The next proposition is the well known P\'olya theorem on positive polynomials over the simplex \cite{P-28}.
\begin{proposition}\label{prop:polya}
Let $h$ be a homogeneous polynomial and positive on the simplex $\Delta_n$. Then, there is a positive integer $N$ such that for all $r\geq N$, the polynomial
\[
(\mathbf e^\mathsf{T}\mathbf x)^rh(\mathbf x)
\]
has positive coefficients.
\end{proposition}

Next, we will generalize Proposition~\ref{prop:polya} to multi-forms over the joint simplex. It will serve as a theoretical foundation for the DNN relaxation methods to be introduced later for \eqref{nonnegative-problem-homo}. The proof is in the spirit of P\'olya \cite{P-28}, see also \cite{HLP-52}.
\begin{proposition}\label{prop:positive}
Let $f$ be a multi-form of degree $d_i$ with respect to each $\mathbf x^{(i)}$ for $i=1,\dots,p$. If $f$ is positive on $\Delta_{n_1}\times\dots\times\Delta_{n_p}$, then
\[
\bigg[\prod_{i=1}^p(\mathbf e^\mathsf{T}\mathbf x^{(i)})^{r_i}\bigg]f(\mathbf x)
\]
is a polynomial with positive coefficients for all sufficiently large $r_i$ with $i\in\{1,\dots, p\}$.
\end{proposition}

\begin{proof}
For any $\gamma=(\gamma^{(1)},\dots,\gamma^{(p)})\in\mathbb N^{n_1}\times\dots\times\mathbb N^{n_p}$, let
\[
\gamma^{(i)}!:=\frac{|\gamma^{(i)}|!}{\prod_{j=1}^{n_i}\gamma^{(i)}_j!}\ \text{for all }i\in\{1,\dots,p\}
\]
and
\[
\gamma!:=\gamma^{(1)}!\dots\gamma^{(p)}!.
\]
If $\gamma\leq\alpha$ (i.e., $\gamma^{(i)}_j\leq\alpha^{(i)}_j$ for all $j\in\{1,\dots,n_i\}$ and $i\in\{1,\dots,p\}$), then
\[
{\alpha^{(i)}\choose\gamma^{(i)}}:=\prod_{j=1}^{n_i}{\alpha^{(i)}_j\choose\gamma^{(i)}_j}\ \text{for all }j\in\{1,\dots,p\}.
\]

Suppose that the polynomial $f$ has the expansion
\[
f(\mathbf x)=\sum_{\alpha\in \Lambda(d_1,\dots,d_p)}\alpha!a_{\alpha}\prod_{i=1}^p{\big(\mathbf x^{(i)}\big)}^{\alpha^{(i)}},
\]
where
\[
\Lambda(d_1,\dots,d_p):=\{\alpha\in\mathbb N^{n_1}\times\dots\times\mathbb N^{n_p}\colon  |\alpha^{(i)}|=d_i\ \text{for all }i=1,\dots,p\}.
\]

Let
\[
\phi(\mathbf x,\mathbf t):=\prod_{i=1}^pt_i^{d_i}\sum_{\alpha\in \Lambda(d_1,\dots,d_p)} a_{\alpha}\prod_{i=1}^p{\mathbf x^{(i)}t_i^{-1}\choose\alpha^{(i)}},
\]
where $\mathbf x\in\mathbb N^{n_1}\times\dots\times\mathbb N^{n_p}$ and $\mathbf t\in\mathbb N_{++}^p$.
Note that for all $i\in\{1,\dots,p\}$, we have
\begin{align*}
t_i^{d_i}{\mathbf x^{(i)}t_i^{-1}\choose\alpha^{(i)}}&=t_i^{d_i}{x_1^{(i)}t_i^{-1}\choose\alpha_1^{(i)}}\dots{x_{n_i}^{(i)}t_i^{-1}\choose\alpha_{n_i}^{(i)}}\\
&=\frac{x_1^{(i)}(x_1^{(i)}-t_i)(x_1^{(i)}-2t_i)\dots (x_1^{(i)}-(\alpha_1^{(i)}-1)t_i)}{\alpha_1^{(i)}!}\\
&\ \ \dots\frac{x_{n_i}^{(i)}(x_{n_i}^{(i)}-t_i)(x_{n_i}^{(i)}-2t_i) \dots (x_{n_i}^{(i)}-(\alpha_{n_i}^{(i)}-1)t_i)}{\alpha_{n_i}^{(i)}!}.
\end{align*}
Therefore, we have that
\begin{equation}\label{eq:limit}
\phi(\mathbf x,\mathbf t)\rightarrow\prod_{i=1}^p\frac{1}{d_i!} f(\mathbf x)\ \text{as }\mathbf t\rightarrow \mathbf 0.
\end{equation}

By the multinomial expansion, we have
\begin{equation}\label{eq:multinomial}
\prod_{i=1}^p(\mathbf e^\mathsf{T}\mathbf x^{(i)})^{r_i}=\prod_{i=1}^p\Big(\sum_{|\gamma^{(i)}|=r_i}\gamma^{(i)}!(\mathbf x^{(i)})^{\gamma^{(i)}}\Big).
\end{equation}
Therefore, multiplying \eqref{eq:multinomial} to the expansion of $f$, we have
\begin{align*}
\bigg[\prod_{i=1}^p(\mathbf e^\mathsf{T}\mathbf x^{(i)})^{r_i}\bigg]f(\mathbf x)
&=\prod_{i=1}^p\Big(\sum_{|\gamma^{(i)}|=r_i}\gamma^{(i)}!(\mathbf x^{(i)})^{\gamma^{(i)}}\Big)\Big(\sum_{\alpha\in \Lambda(d_1,\dots,d_p)}\alpha!a_{\alpha}\prod_{i=1}^p\big(\mathbf x^{(i)}\big)^{\alpha^{(i)}}\Big)
\\
&=\sum_{\alpha\in \Lambda(d_1,\dots,d_p)}\sum_{|\gamma^{(1)}|=r_1}\dots\sum_{|\gamma^{(p)}|=r_p}\alpha!a_{\alpha}
\Big(\prod_{i=1}^p\gamma^{(i)}!\Big)\prod_{i=1}^p\big(\mathbf x^{(i)}\big)^{\alpha^{(i)}+\gamma^{(i)}}\\
&=\frac{\prod_{i=1}^pd_i!\prod_{i=1}^pr_i!}{\prod_{i=1}^ps_i!}\sum_{\kappa\in\Lambda(s_1,\dots,s_p)}\prod_{i=1}^p\big(\mathbf x^{(i)}\big)^{\kappa^{(i)}}
\left[\sum_{\alpha\in \Lambda(d_1,\dots,d_p)} \Big(\prod_{i=1}^p\kappa^{(i)}!\Big)a_{\alpha}\prod_{i=1}^p{\kappa^{(i)}\choose\alpha^{(i)}} \right]
\\
&=\frac{\prod_{i=1}^pd_i!\prod_{i=1}^pr_i!}{\prod_{i=1}^ps_i!}\sum_{\kappa\in\Lambda(s_1,\dots,s_p)}\Big(\prod_{i=1}^p\kappa^{(i)}!\Big)
\big(\mathbf x^{(i)}\big)^{\kappa^{(i)}}
\left[ \sum_{\alpha\in \Lambda(d_1,\dots,d_p)}a_{\alpha}\prod_{i=1}^p{\kappa^{(i)}\choose\alpha^{(i)}} \right]
\\
&=\frac{\prod_{i=1}^pd_i!\prod_{i=1}^pr_i!}{\prod_{i=1}^ps_i!}\prod_{i=1}^ps_i^{d_i}\sum_{\kappa\in\Lambda(s_1,\dots,s_p)}\Big(\prod_{i=1}^p\kappa^{(i)}!\Big)\big(\mathbf x^{(i)}\big)^{\kappa^{(i)}}\phi(\kappa/\mathbf s,1/\mathbf s),
\end{align*}
where the third equality follows from the fact that for all $i\in\{1,\dots,p\}$
\begin{align*}
& \frac{|\alpha{(i)}|!}{\alpha^{(i)}_1!\dots\alpha^{(i)}_{n_i}!}\frac{|\gamma{(i)}|!}{\gamma^{(i)}_1!\dots\gamma^{(i)}_{n_i}!} \;=\; \frac{d_i!r_i!}{(\alpha^{(i)}_1+\gamma^{(i)}_1)!\dots(\alpha^{(i)}_{n_i}+\gamma^{(i)}_{n_i})!}\prod_{j=1}^{n_i}{\alpha^{(i)}_j+\gamma^{(i)}_j\choose\alpha^{(i)}_j}
\\[5pt]
&=\; \frac{d_i!r_i!}{\kappa^{(i)}_1!\dots \kappa^{(i)}_{n_i}!}\prod_{j=1}^{n_i}{\kappa^{(i)}_j\choose\alpha^{(i)}_j}
\;\; =\;\;
\frac{d_i!r_i!}{|\kappa^{(i)}|!}\frac{|\kappa^{(i)}|!}{\kappa^{(i)}_1!\dots \kappa^{(i)}_{n_i}!}\prod_{j=1}^{n_i}{\kappa^{(i)}_j\choose\alpha^{(i)}_j}
\\[5pt]
&=\;  \frac{d_i!r_i!}{s_i!}\kappa^{(i)}!{\kappa^{(i)}\choose \alpha^{(i)}}.
\end{align*}
In the above,
\[
\kappa/\mathbf s:=(\kappa^{(1)}/s_1,\dots,\kappa^{(p)}/s_p)
\]
and
\[
s_i:=d_i+r_i\ \text{for all }i\in\{1,\dots,p\}.
\]
Since $f$ is positive over the joint simplex $\Delta_{n_1}\times\dots\times\Delta_{n_p}$ which is compact, we have that there exists $\mu>0$ such that
\[
f(\mathbf x)\geq\mu>0\ \text{for all }\mathbf x\in \Delta_{n_1}\times\dots\times\Delta_{n_p}.
\]
Obviously, $\kappa/\mathbf s\in \Delta_{n_1}\times\dots\times\Delta_{n_p}$. Thus for sufficiently large $\mathbf s$ (of course component-wisely),
by using \eqref{eq:limit},
we have
\[
\phi(\kappa/\mathbf s,1/\mathbf s)\geq\frac{\mu}{2}>0.
\]
Consequently, the result follows.
\end{proof}

The complexity of this Positivstellensatz can be investigated, as in\cite{NS-07,PR-01}.
But we will leave it for the future research since this article is focused on the zero-th order relaxation.

\subsection{DNN relaxation}\label{sec:sos}
In this section, we will introduce a doubly nonngeative (DNN) relaxation method for solving problem~\eqref{nonnegative-problem-homo}.

Let $\mathbf z = (z_1,\dots,z_n)^\mathsf{T}$, and
\[
\mathbf z^{[s]}:=\big(z_1^s,z_1^{s-1}z_2,z_1^{s-1}z_3,\dots,z_1^{s-2}z_2^2,z_1^{s-2}z_2z_3,\dots,z_2^s,\dots,z_n^s\big)^\mathsf{T}
\]
be the monomial basis of degree $s$ in $n$ variables. The order is the lexicographic order and $z_1\succ z_2\succ \dots\succ z_n$.
Note that the length of $\mathbf z^{[s]}$ is
\[
\nu(s,n):={n+s-1\choose s}.
\]

Let $\tau=(\tau_1,\dots,\tau_p)\in\mathbb Z^p_+$, $\mathbf x\in\mathbb R^{n_1}\times\dots\times\mathbb R^{n_p}$, and
\[
\mathbf x^{[\tau]}:=\big(\mathbf x^{(1)}\big)^{[\tau_1]}\otimes\dots\otimes \big(\mathbf x^{(p)}\big)^{[\tau_p]}.
\]
The monomials are ordered in the lexicographic order with $\mathbf x^{(1)}\succ \dots\succ \mathbf x^{(p)}$ for the groups of variables.
Let
\[
\nu(\tau,n_1,\dots,n_p):=\prod_{j=1}^p\nu(\tau_j,n_j),
\]
and $A_\alpha\in\mathbb R^{\nu(\tau,n_1,\dots,n_p)\times\nu(\tau,n_1,\dots,n_p)}$ be the coefficient matrix of $\mathbf x^{[\tau]}\big(\mathbf x^{[\tau]}\big)^\mathsf{T}$ in the standard basis $\mathbf x^{[2\tau]}$, i.e.,
\begin{equation}\label{coefficient-moment}
\mathbf x^{[\tau]}\big(\mathbf x^{[\tau]}\big)^\mathsf{T}=\sum_{\alpha\in \mathbb N^{n_1}_{2\tau_1}\times\dots\times\mathbb N^{n_p}_{2\tau_p}}A_\alpha \mathbf x^\alpha,
\end{equation}
where $\mathbb N^n_m:=\{\gamma\in\mathbb N^n : \gamma_1+\dots+\gamma_n=m\}$.

Before stating the DNN relaxation problem, we first give a simple observation on the nonnegativity of moment sequences.
\begin{proposition}[Nonnegativity Equivalence]\label{prop:monomial}
Let all notation be as above. Then, the coefficient matrices in the set $\{A_{\alpha}\}$ are nonnegative and orthogonal to each other, and thus
\begin{equation}\label{eq:moment}
\mathbf y\in\mathbb R^{\nu(\mathbf d,n_1,\dots,n_p)}_+ \ \text{if and only if }M(\mathbf y):=\sum_{\alpha\in \mathbb N^{n_1}_{d_1}\times\dots\times\mathbb N^{n_p}_{d_p}}A_\alpha y_\alpha \geq 0.
\end{equation}
\end{proposition}

\begin{proof}
According to the definition, each $A_\alpha$ is a nonnegative matrix. Therefore, the necessity is obvious.
The sufficiency follows from the fact that
\[
\langle A_\alpha, A_\gamma\rangle = 0
\]
for all $\alpha\neq \gamma$, and
\[
\sum_{\alpha\in \mathbb N^{n_1}_{d_1}\times\dots\times\mathbb N^{n_p}_{d_p}}A_\alpha = E,
\]
where $E$ is
the matrix of all ones.
\end{proof}

Denote $\mathbf d:=2\tau=(2\tau_1,\dots,2\tau_p)$.
Let
$\mathbf f\in\mathbb R^{\nu(\mathbf d,n_1,\dots,n_p)}$ be the coefficient vector of the polynomial $f(\mathbf x^{(1)},\dots,\mathbf x^{(p)})$ in the standard basis $\mathbf x^{[\mathbf d]}$, and let $\mathbf g\in\mathbb R^{\nu(\mathbf d,n_1,\dots,n_p)}$ be that for the polynomial $g(\mathbf x):=\prod_{j=1}^p\big[\big(\mathbf x^{(j)}\big)^\mathsf{T}\mathbf x^{(j)}\big]^{\tau_j}$.

The basic idea of the SOS relaxation in \cite{L-01} is by relaxing the rank characterization of a moment vector $\mathbf y\in\mathbb R^{\nu(\mathbf d,n_1,\dots,n_p)}$. Without the nonnegativity constraint, it is classically relaxed as $M(\mathbf y)\succeq 0$, i.e., the positive semidefiniteness of the moment matrix, see \cite{L-01,N-12,NZ-16}. It can be shown that the dual problem under this method is an SDP problem
obtained by representing a polynomial as a sum of squares (SOS).
Therefore, this relaxation method is usually referred to as the \textit{SOS relaxation}.
With Proposition~\ref{prop:monomial}, a moment vector generated by a nonnegative vector is then naturally relaxed as $M(\mathbf y)\succeq 0$ and $M(\mathbf y)\geq 0$, i.e., the moment matrix is both positive semidefinite and component-wisely nonnegative. A matrix
that is both positive semidefinite and component-wisely nonnegative is
said to be \textit{doubly nonnegative}.

Naturally, a standard \textit{doubly nonnegative (DNN) relaxation} of problem \eqref{nonnegative-problem-homo} is
\begin{equation}\label{relaxation-primal}
\begin{array}{rrl}f_{\text{dnn}}:=& \min&\langle \mathbf f,\mathbf y\rangle
\\[3pt]
&\text{s.t.}& M(\mathbf y)\succeq \mathbf 0,\\[3pt]
&& M(\mathbf y)\geq \mathbf 0,\\[3pt]
&&\langle \mathbf g,\mathbf y\rangle =1, \ \mathbf y\in\mathbb R^{\nu(\mathbf d,n_1,\dots,n_p)}, \end{array}
\end{equation}
and the dual of which is
\begin{equation}\label{relaxation-dual}
\begin{array}{rl} \max&\gamma
\\[3pt]
\text{s.t.}&\mathbf f-\gamma \mathbf g\in\Sigma^+_{\mathbf d,n_1,\dots,n_p}, \end{array}
\end{equation}
where
\[
\Sigma^+_{\mathbf d,n_1,\dots,n_p}:=\big\{\mathbf h\colon h(\mathbf x)\in\mathbb R[\mathbf x]_{\mathbf d},\  h(\mathbf x)=\big(\mathbf x^{[\tau]}\big)^\mathsf{T}(S+T)(\mathbf x^{[\tau]})\ \text{for some }S\succeq \mathbf 0\ \text{and }T\geq \mathbf 0\big\}.
\]
Here $\mathbb R[\mathbf x]_{\mathbf d}\subset \mathbb R[\mathbf x]$ is the set of multi-forms being homogeneous of degree $d_i$ with respect to $\mathbf x^{(i)}$ for all $i\in\{1,\dots,p\}$.
Note that the cone of sums of squares
\[
\Sigma_{\mathbf d,n_1,\dots,n_p}=\big\{\mathbf h\colon h(\mathbf x)\in\mathbb R[\mathbf x]_{\mathbf d},\  h(\mathbf x)=\big(\mathbf x^{[\tau]}\big)^\mathsf{T}S(\mathbf x^{[\tau]})\ \text{for some }S\succeq \mathbf 0\big\}
\]
is strictly contained in $\Sigma^+_{\mathbf d,n_1,\dots,n_p}$. If there is no confusion, we sometimes will write $h(\mathbf x)\in \Sigma^+_{\mathbf d,n_1,\dots,n_p}$ for a multi-form $h(\mathbf x)$, meaning its coefficient vector $\mathbf h\in \Sigma^+_{\mathbf d,n_1,\dots,n_p}$.

The above DNN relaxation, together with Proposition~\ref{prop:positive}, motivates a hierarchy of DNN relaxations for the
optimization problem~\eqref{nonnegative-problem-homo}.
\begin{proposition}\label{prop:convergence-dnn}
Let $\eta\in\mathbb N^p$ and $\gamma_{\eta}$ be the optimal value of the following problem
\begin{equation}\label{eq:hierarchy-dnn}
\gamma_{\eta}:=\max\bigg\{\gamma\colon \prod_{i=1}^p(\mathbf e^\mathsf{T}\mathbf x^{(i)})^{2\eta_i}(f(\mathbf x)-\gamma g(\mathbf x))\in\Sigma^+_{\mathbf d+2\eta,n_1,\dots,n_p}\bigg\}.
\end{equation}
Then
\begin{equation}\label{eq:opt-dnn-relax}
f_{\text{{\rm dnn}}}\leq \gamma_{\eta}\leq f_{\min},\ \text{and }\gamma_{\eta}\leq\gamma_{\overline{\eta}}\ \text{whenever }\eta\leq\overline{\eta},
\end{equation}
and
\[
\gamma_{\eta}\rightarrow f_{\min}\ \text{as }\min\{\eta_i\colon i=1,\dots,p\}\rightarrow\infty.
\]
\end{proposition}

\begin{proof}
The relations in \eqref{eq:opt-dnn-relax} follows directly from the fact that each $\mathbf e^\mathsf{T}\mathbf x^{(i)}$ is a polynomial with positive coefficients.

Given an arbitrary $\epsilon>0$, we know that the multi-form $f(\mathbf x)-(f_{\min}-\epsilon) g(\mathbf x)$ is positive on the nonnegative multi-sphere. Since $f(\mathbf x)-(f_{\min}-\epsilon) g(\mathbf x)$ is a multi-form, it is still positive on the joint simplex. Thus, it follows from Proposition~\ref{prop:positive} that there are positive integers $r_i$'s such that
\[
\prod_{i=1}^p(\mathbf e^\mathsf{T}\mathbf x^{(i)})^{2\eta_i}(f(\mathbf x)-(f_{\min}-\epsilon) g(\mathbf x))\in \Sigma^+_{\mathbf d+2\eta,n_1,\dots,n_p}
\]
for all $\eta\geq \mathbf r$. Therefore, for all $\eta\geq\mathbf r$,
\[
f_{\min}-\epsilon\leq \gamma_{\eta}\leq f_{\min}.
\]
The conclusion thus follows.
\end{proof}

Proposition~\ref{prop:convergence-dnn} gives the global convergence of the hierarchy of DNN relaxations (cf.\ \eqref{eq:hierarchy-dnn}) for
the problem~\eqref{nonnegative-problem-homo}, parallel to that of SOS relaxations (cf.\ \cite[Theorem~3.4]{L-01}).
Problem~\eqref{relaxation-dual} is the \textit{zero-th order} DNN relaxation, i.e., $\eta=\mathbf 0$ in \eqref{eq:hierarchy-dnn}.

In the following, some
properties on the two matrix optimization problems \eqref{relaxation-primal} and \eqref{relaxation-dual} will be investigated.

\begin{lemma}\label{lem:strict}
There exists a $\mathbf y\in\mathbb R^{\nu(\mathbf d,n_1,\dots,n_p)}$ such that $M(\mathbf y)\succ\mathbf 0$ and $M(\mathbf y)>\mathbf 0$, i.e., the linear conic problem \eqref{relaxation-primal} is strictly feasible.
\end{lemma}

\begin{proof}
Let $\lambda$ be the Lebesgue measure on $\mathbb S^{n_1-1}\times\dots\times\mathbb S^{n_p-1}$.
Let $\mu$ be the normalized standard measure over the nonnegative multi-sphere $S:=(\mathbb R_+^{n_1}\cap\mathbb S^{n_1-1})\times\dots\times(\mathbb R_+^{n_p}\cap\mathbb S^{n_p-1})$, also known as the uniform probability measure on $\mathbb S^{n_1-1}_+\times\dots\times\mathbb S^{n_p-1}_+$, defined as
\[
\mu(A):=\frac{1}{\lambda(S)}\lambda(A\cap S)\ \text{for any Borel set }A.
\]
Define
\[
y_{\alpha}:=\int\mathbf x^\alpha \operatorname{d}\mu\ \ \text{for all }\alpha\in \mathbb N^{n_1}_{d_1}\times\dots\times\mathbb N^{n_p}_{d_p}
\]
to be the truncated moment sequence of $\mu$. It is obvious that $\mathbf y> \mathbf 0$, and
\[
\langle\mathbf g,\mathbf y\rangle=\int g(\mathbf x)\operatorname{d}\mu = 1,
\]
since $g(\mathbf x)\equiv 1$ over the support $S$ of $\mu$.

For any $f(\mathbf x)\in\mathbb R[\mathbf x]_{\tau}$, we have
\[
\mathbf f^\mathsf{T}M(\mathbf y)\mathbf f = \int f(\mathbf x)^2\operatorname{d}\mu.
\]
Since the support of $\mu$ is the nonnegative orthant part of the multi-sphere, if $\mathbf f^\mathsf{T}M(\mathbf y)\mathbf f =0$, we then must have that
\[
f(\mathbf x)=0\ \text{for all }\mathbf x\in S:=(\mathbb R_+^{n_1}\cap\mathbb S^{n_1-1})\times\dots\times(\mathbb R_+^{n_p}\cap\mathbb S^{n_p-1}).
\]
Since $f$ is multi-homogeneous, we immediately have that
\[
f(\mathbf x)=0\ \text{for all }\mathbf x\in \mathbb R_+^{n_1}\times\dots\times \mathbb R_+^{n_p}.
\]
Note that $\mathbb R_+^{n_1}\times\dots\times \mathbb R_+^{n_p}$ is a set with the Zariski closure being the whole space $\mathbb R^{n_1}\times\dots\times \mathbb R^{n_p}$. We conclude that $f\equiv 0$. Thus, the matrix $M(\mathbf y)$ is positive definite.
\end{proof}

\begin{lemma}\label{lem:strict-dual}
There exists a scalar $\gamma$, a matrix $S\succ \mathbf 0$, and a matrix $T>\mathbf 0$ such that $f(\mathbf x)-\gamma g(\mathbf x)=\big(\mathbf x^{[\tau]}\big)^\mathsf{T}(S+T)(\mathbf x^{[\tau]})$, i.e., the linear conic problem \eqref{relaxation-dual} is strictly feasible.
\end{lemma}

\begin{proof}
Note that there exists a nonnegative diagonal matrix $D$ such that
\[
g(\mathbf x) = \big(\mathbf x^{[\tau]}\big)^\mathsf{T}D(\mathbf x^{[\tau]})
\]
and the minimum diagonal element being one. Thus, $D\succ\mathbf 0$.
The result follows immediately if a sufficiently small $\gamma<0$ is chosen.
\end{proof}

\begin{proposition}\label{prop:solve}
Both \eqref{relaxation-primal} and \eqref{relaxation-dual} are solvable, and there is no duality gap.
\end{proposition}

\begin{proof}
Both  \eqref{relaxation-primal} and \eqref{relaxation-dual} have strictly feasible solutions by  Lemmas~\ref{lem:strict} and \ref{lem:strict-dual} respectively.
 The conclusion then follows from standard duality theory for linear conic optimization problems (cf.\ \cite{BT-01}).
\end{proof}

\begin{proposition}[Exact Relaxation]\label{prop:exact}
Let $d_i=2\tau_i$ for all $i=1,\dots,p$. If \eqref{relaxation-primal} has an optimal solution $\mathbf y^*$ such that
\begin{equation}\label{rank-one}
\operatorname{rank}(M(\mathbf y^*)) = 1,
\end{equation}
then the relaxation is tight, i.e., $f_{\min} = f_{\text{\rm dnn}}$, and an optimal solution for \eqref{nonnegative-problem} can be extracted from $\mathbf y^*$.
\end{proposition}

\begin{proof}
It follows from \cite{Curto-Fialkow:rank,PS-03} that $\mathbf y^*$ is a monomial vector in this situation. Let
\[
M(\mathbf y^*) = \mathbf x^{[\tau]}_*\big(\mathbf x^{[\tau]}_*\big)^\mathsf{T}
\]
with $\mathbf x_*=(\mathbf x^{(1)}_*,\dots,\mathbf x^{(p)}_*)$. Then, we have from $M(\mathbf y^*)\geq \mathbf 0$ that
\[
\mathbf x^{(i)}_*\geq \mathbf 0\;\; \text{or }\;\; \mathbf x^{(i)}_*\leq \mathbf 0
\]
for each $i=1,\dots,p$. Since each $d_i$ is even, the monomial vector $\mathbf z^*$ with
\[
M(\mathbf z^*) = \mathbf w^{[\tau]}_*\big(\mathbf w^{[\tau]}_*\big)^\mathsf{T}\ \text{and }\mathbf w_*=(|\mathbf x^{(1)}_*|,\dots,|\mathbf x^{(p)}_*|)
\]
satisfies that
\[
\mathbf y^*=\mathbf z^*.
\]
Therefore, the results follow.
\end{proof}

We will see from later numerical experiments that \eqref{rank-one} is a \textit{typical property}, i.e., it holds 
with high probability
if we randomly generate $f$ from a continuous probability distribution.

\subsection{DNN reformulation}\label{sec:dnn}
In this section, we formulate \eqref{relaxation-primal} as
a linear optimization problem over the cone of doubly nonnegative matrices more explicitly.
We shall replace
the variable vector $\mathbf y$ by exploiting
 the hidden constraints on the matrix $M(\mathbf y)$.
We have already show that the $\nu(\mathbf d,n_1,\dots,n_p) $ matrices (cf.\ \eqref{coefficient-moment})
\[
 A_\alpha : \alpha\in \mathbb N^{n_1}_{2\tau_1}\times\dots\times\mathbb N^{n_p}_{2\tau_p}
\]
are orthogonal to each other. Let
\[
\big\{ B_i : 1\leq i\leq \mu(\mathbf d,n_1,\dots,n_p):=\nu(\tau,n_1,\dots,n_p)(\nu(\tau,n_1,\dots,n_p)+1)/2 - \nu(\mathbf d,n_1,\dots,n_p) \big\}
\]
be the set of matrices that are orthogonal to each other such that
\[
\big\{A_\alpha : \alpha\in \mathbb N^{n_1}_{2\tau_1}\times\dots\times\mathbb N^{n_p}_{2\tau_p}\big\}\cup\big\{B_i :  i=1,\dots,\mu(\mathbf d,n_1,\dots,n_p) \big\}
\]
forms an orthogonal basis of the space of $\nu(\tau,n_1,\dots,n_p)\times \nu(\tau,n_1,\dots,n_p)$ real symmetric matrices. 
Let
\[
\mathbf w \in\mathbb R^{\nu(\mathbf d,n_1,\dots,n_p)}\text{ with } w_\alpha = \langle A_\alpha,A_\alpha\rangle\ \text{for all }\alpha.
\]
Then
the problem \eqref{relaxation-primal} can be equivalently reformulated as
\begin{equation}\label{relaxation-dnn}
\begin{array}{rrl}f_{\text{dnn}}:=& \min&\big\langle\sum_{\alpha\in\mathbb N^{n_1}_{2\tau_1}\times\dots\times\mathbb N^{n_p}_{2\tau_p}} \frac{f_\alpha}{w_\alpha}A_\alpha, X\big\rangle
\\[5pt]
&\text{s.t.}& \langle B_i,X\rangle = 0,\ i=1,\dots,\mu(\mathbf d,n_1,\dots,n_p),
\\[5pt]
&&\big\langle\sum_{\alpha\in\mathbb N^{n_1}_{2\tau_1}\times\dots\times\mathbb N^{n_p}_{2\tau_p}} \frac{g_\alpha}{w_\alpha}A_\alpha, X\big\rangle =1,
\\[8pt]
& & X\succeq \mathbf 0,\ X\geq\mathbf 0. \end{array}
\end{equation}
The optimization problem \eqref{relaxation-dnn} is classified as a doubly nonnegative (DNN) problem, since it requires the matrix variable $X$ to be both positive semidefinite and component-wisely nonnegative. As a linear conic problem, it can be reformulated as a standard semidefinite programming (SDP) problem introducing a
new variable $Y$ and add the constraints
that $X-Y=0$ so that the original doubly nonnegative conic constraint
can be replaced by $X\succeq 0$ and $Y\geq 0$. However, this
reformulation introduces too many new equality constraints which
not only make the resulting standard SDP problem computationally much
more expensive to solve but we are also likely to encounter numerical difficulties
when solving
the standard SDP reformulation since it is likely to be constraint degenerate (cf.\ \cite{ZST-10}). 

The next table gives some 
information
on the sizes of the DNN relaxation problem \eqref{relaxation-dnn} for different sizes $\mathbf d$ and $n_1,\dots,n_p$. When $d_i$ is odd, we use the technique in Section~\ref{sec:odd} to transform it into the standard formulation involving only even orders. In this table, \textbf{$\#\ \text{eq.}$} means the number of equality constraints, and \textbf{dim} means the dimension of the matrix variable. Except the first case of a quartic polynomial in $100$ variables, all the other cases are almost hopeless to solve at present \cite{YST-15}. On the other hand, all the cases are tensors with small to moderate dimensions, showing the difficulty of the problem \eqref{nonnegative-problem} from another perspective.
\begin{table}[h!]
\centering
\begin{tabular}{|c|c|}
\hline
$(\mathbf d,n_1,\dots,n_p)$ : (\# \text{eq.}; \text{dim}) & $(\mathbf d,n_1,\dots,n_p)$ : (\# \text{eq.}; \text{dim}) \\ \hline
   (4, 100)  : (8,332,501; 5,050)  & (4,150) : (42,185,626;  11,325)        \\\hline
  ((2,2), 100,100)  : (  24,502,501; 10,000)  & ((2,3),50,20) : (53,158,876;  11,550)        \\\hline
  ((2,2,2), 20,20,20)  : ( 22,743,001; 8,000)  & ((2,2,3),15,10,10) : (42,403,351;  9,900)        \\\hline
  ((2,2,2,2), 10,10,10,10)  : (40,854,376; 10,000)  & ((2,2,2,3),6,6,6,8) : ( 42,659,866;  9,720)\\ \hline
\end{tabular}
\caption{($\mathbf d,n_1,\dots,n_p$) : (number of equations; dimension of the matrix space) of \eqref{relaxation-dnn} for several $\mathbf d$'s and $(n_1,\dots,n_p)$'s}
\end{table}

\subsection{Worst case approximation bound}\label{sec:bound}

In this section, we present a worst case approximation bound for $f_{\text{dnn}}$.

Given a positive integer $n$, define the matrix $\Theta_n$ by
\[
\Theta_n:=\int_{\mathbb S^{n-1}_+}\mathbf x^{[n]}\big(\mathbf x^{[n]}\big)^\mathsf{T}\operatorname{d}\mu(\mathbf x),
\]
where $\mu(\mathbf x)$ is the uniform probability measure on $\mathbb S^{n-1}_+$.
It is easy to see that $\Theta_n$ is positive definite, since the set $\mathbb S^{n-1}_+$ is of dimension $n-1$ and the monomial vector $\mathbf x^{[n]}$ consists of homogeneous monomials. Let
\[
\delta_{n_1,\dots,n_p}:=\prod_{i=1}^p\sqrt{\lambda_{\min}(\Theta_{n_i})},
\]
where $\lambda_{\min}(\Theta_{n_i})$ is the smallest eigenvalue of the matrix $\Theta_{n_i}$.
Since each $\Theta_{n_i}$ is positive definite, we have that $\delta_{n_1,\dots,n_p}>0$.

 Since the set $\mathbb S^{n-1}_+$ is involved in this article instead of $\mathbb S^{n-1}$, $\lambda_{\min}(\Theta_{n_i})$ is different from those given in \cite[Table~1]{N-12}. For example,
\[
\Theta_2=\frac{1}{8\pi}\begin{bmatrix}3\pi& 4&\pi\\ 4&2\pi& 4\\ \pi&4& 3\pi\end{bmatrix}.
\]
Consequently, $\delta_{2}=\sqrt{\lambda_{\min}(\Theta_2)}=0.4849$, which is different from $0.5$ in \cite{N-12} with respect to $\mathbb S^{n-1}$.

Let $f_{\max}$ and $f_{\min}$ be the maximum and minimum values of the objective function $f$ over the feasible set of problem \eqref{nonnegative-problem}.
We then have the next proposition, whose proof is almost the same as that in \cite[Theorem~3.4]{N-12}.
\begin{proposition}\label{prop:bound}
Suppose that $n_i\geq d_i$ for all $i\in\{1,\dots,p\}$ and all notation are as above. Then we have that
\begin{equation}\label{eq:bound}
1 \; \leq\;
\frac{f_{\max}-f_{\text{dnn}}}{f_{\max}-f_{\min}}
\;\leq\; \frac{1}{\delta_{d_1,\dots,d_p}}\sqrt{{n_1\choose d_1}\dots{n_p\choose d_p}}.
\end{equation}
\end{proposition}

With $\delta_2$ computed as above, we have that for a biquadratic form over the intersection of the multi-sphere and the nonnegative orthant
\[
1 \;\leq\; \frac{f_{\max}-f_{\textit{dnn}}}{f_{\max}-f_{\min}}\; \leq\;
 4.2535\sqrt{{n_1\choose 2}{n_2\choose 2}}.
\]
The upper bound is slightly different from that with respect to the multi-sphere, see \cite[Corollary~3.5]{N-12}.

If the polynomial is sparse, i.e., with fewer terms in its polynomial expansion, then an improved worst case approximation bound in terms of the number of monomials $\Omega(f)$ can be derived as in \cite[Section~4]{N-12}. In particular, if the polynomial is a monomial or the number of monomials is bounded by a constant, then a constant worst case approximation bound, independent of the problem dimensions, can be given.

\subsection{Solution extraction for even order tensors}\label{sec:solution-even}
Let $\mathbf y^*$ be an optimal solution for \eqref{relaxation-primal}. By Proposition~\ref{prop:monomial}, $\mathbf y^*\geq \mathbf 0$.
Let
\[
y^*_{2\gamma}:=\max\{y^*_{2\mu} : \mathbf x^\mu\in\mathbf x^{[\tau]}\}.
\]
Since the set $\{y^*_{2\mu} : \mathbf x^\mu\in\mathbf x^{[\tau]}\}$ forms the diagonal elements of the positive semidefinite matrix $M(\mathbf y^*)$ and $\mathbf y^*\neq 0$, we have that
\[
y^*_{2\gamma}>0.
\]
Denote
\[
\gamma:=(\gamma^1,\dots,\gamma^p)
\]
with
\[
\gamma^i:=(\gamma^i_1,\dots,\gamma^i_{n_i})
\]
for all $i=1,\dots,p$. Then $\gamma^i\neq\mathbf 0$ for all $i=1,\dots,p$. Let
\[
\gamma^i_{k_i}:=\max\{\gamma^i_1,\dots,\gamma^i_{n_i}\}.
\]
Define
\begin{eqnarray*}
\mathbf z^{(i)}_*  &:=&(y^*_{\gamma+(\gamma^1,\dots,\gamma^{i-1},\gamma^i-\mathbf e^{(i)}_{k_i}+\mathbf e^{(i)}_1,\gamma^{i+1},\dots,\gamma^p)},\dots,y^*_{\gamma+(\gamma^1,\dots,\gamma^{i-1},\gamma^i-\mathbf e^{(i)}_{k_i}+\mathbf e^{(i)}_{n_i},\gamma^{i+1},\dots,\gamma^p)})^\mathsf{T}
\\[5pt]
\mathbf x^{(i)}_* &:=&
|\mathbf z^{(i)}_*|/\|\mathbf z^{(i)}_*\| \ \ \text{for all }i=1,\dots,p.
\end{eqnarray*}
The approximation solution is then
\[
\mathbf x_*=(\mathbf x^{(1)}_*,\dots,\mathbf x^{(p)}_*),
\]
and the approximation value is
\[
f_{\text{app}}:=f(\mathbf x^{(1)}_*,\dots,\mathbf x^{(p)}_*).
\]
If $\operatorname{rank}(M(\mathbf y^*))=1$, then it holds that (cf.\ Proposition~\ref{prop:exact})
\[
M(\mathbf y^*)=\mathbf x^{[\tau]}_*\big(\mathbf x^{[\tau]}_*\big)^\mathsf{T}.
\]

\subsection{Solution extraction for odd order tensors}\label{sec:solution-odd}
Let $\mathbf y^*$ be an optimal solution for \eqref{relaxation-primal}. Suppose that the tensor space is
$\operatorname{Sym}(\otimes^{\alpha_1}\mathbb R^{n_1})\otimes\dots\otimes\operatorname{Sym}(\otimes^{\alpha_p}\mathbb R^{n_p})$, and without loss of generality that $\alpha_1,\dots,\alpha_q$ are odd for some $q\leq p$. Let $\mathbf d=(\alpha_1+1,\dots,\alpha_q+1,\alpha_{q+1},\dots,\alpha_p)$. By the scheme in Section~\ref{sec:odd}, we have that
\[
\mathbf y^*\in\mathbb R^{\nu(\mathbf d,n_1+1,\dots,n_q+1,n_{q+1},\dots,n_p)}.
\]

Let
\[
y^*_{\gamma}:=\max\big\{y^*_{\mu} : \mu=\big((\mu^1,1),\dots,(\mu^q,1),\mu^{q+1},\dots,\mu^p\big)\ \text{with }\mu^i\in\mathbb N^{n_i}_{\alpha_i}\big\}.
\]
If $y^*_\gamma=0$, it follows from Section~\ref{sec:non-app} that zero is the best approximation solution, since in this case the optimal value of \eqref{relaxation-primal} is zero. In the following, we assume that
\[
y^*_{\gamma}>0.
\]
Denote
\[
\gamma:=(\gamma^1,\dots,\gamma^p)
\]
with
\[
\gamma^i:=(\gamma^i_1,\dots,\gamma^i_{n_i},1)
\]
for all $i=1,\dots,q$, and
\[
\gamma^i:=(\gamma^i_1,\dots,\gamma^i_{n_i})
\]
for all $i=q+1,\dots,p$. Let
\[
\gamma^i_{k_i}:=\max\{\gamma^i_1,\dots,\gamma^i_{n_i}\}.
\]
Define
\begin{eqnarray*}
\mathbf z^{(i)}_* &:=& (y^*_{\gamma+(\gamma^1,\dots,\gamma^{i-1},\gamma^i-\mathbf e^{(i)}_{k_i}+\mathbf e^{(i)}_1,\gamma^{i+1},\dots,\gamma^p)},\dots,y^*_{\gamma+(\gamma^1,\dots,\gamma^{i-1},\gamma^i-\mathbf e^{(i)}_{k_i}+\mathbf e^{(i)}_{n_i+1},\gamma^{i+1},\dots,\gamma^p)})^\mathsf{T}\\[5pt]
\tilde{\mathbf x}^{(i)}_* &:=&
|\mathbf z^{(i)}_*|/\|\mathbf z^{(i)}_*\| \ \ \text{for all }i=1,\dots,q.
\end{eqnarray*}
and
\begin{eqnarray*}
\mathbf z^{(i)}_* &:=& (y^*_{\gamma+(\gamma^1,\dots,\gamma^{i-1},\gamma^i-\mathbf e^{(i)}_{k_i}+\mathbf e^{(i)}_1,\gamma^{i+1},\dots,\gamma^p)},\dots,y^*_{\gamma+(\gamma^1,\dots,\gamma^{i-1},\gamma^i-\mathbf e^{(i)}_{k_i}+\mathbf e^{(i)}_{n_i},\gamma^{i+1},\dots,\gamma^p)})^\mathsf{T}\\[5pt]
 \mathbf x^{(i)}_* &:=&
 |\mathbf z^{(i)}_*|/\|\mathbf z^{(i)}_*\| \ \ \text{for all }i=q+1,\dots,p.
\end{eqnarray*}
The approximation solution for the extended problem is then
\[
\tilde{\mathbf x}_*=(\tilde{\mathbf x}^{(1)}_*,\dots,\tilde{\mathbf x}^{(q)}_*,\mathbf x^{(q+1)}_*,\dots,\mathbf x^{(p)}_*).
\]
Let
\[
\tilde{\mathbf x}^{(i)}_*=(\mathbf x^{(i)},t_i) \ \text{for }i=1,\dots,q.
\]
For $i=1,\dots,q$, if $t_i\neq 1$, then we take
\[
\mathbf x^{(i)}_*:=\mathbf x^{(i)}/\|\mathbf x^{(i)}\|.
\]
Otherwise, we conclude that the
best approximating nonnegative rank-one tensor is the zero tensor.

\section{Numerical Experiments}\label{sec:numerical}
In this section, we present some preliminary numerical experiments for solving problem~\eqref{nonnegative-problem} via the DNN relaxation method developed in Section~\ref{sec:homo}.
All the tests were conducted on a Dell PC with 4GB RAM  and 3.2GHz CPU running 64bit Windows operation system. All codes were written in {\sc Matlab} with some subroutines in
C++. All the linear matrix conic problems were solved by SDPNAL+ \cite{YST-15}.
\subsection{Best nonnegative rank-one approximation of tensors}
\label{sec:numerical-best}

In this section, computational results for numerous instances of the best nonnegative rank-one approximation of tensors will be presented.
The tested tensors are taken from the literature.

Given a tensor $\mathcal A$, we use $f_{\text{dnn}}$ to denote the optimal value of the corresponding DNN relaxation problem.
The approximation solution $\mathbf x$ of problem \eqref{nonnegative-problem} is extracted according to Sections~\ref{sec:solution-even} and \ref{sec:solution-odd}. Then $\lambda\mathbf x^{\otimes \mathbf d}$ with $\lambda$ giving by \eqref{lambda-value} is the best nonnegative rank-one approximation found. Therefore, $f_{\text{app}}:=\lambda$ is the approximate optimal value of \eqref{nonnegative-problem} found by the method.
We use the relative approximation error
\[
\textbf{appr}(\mathcal A):=\frac{|f_{\text{dnn}}-f_{\text{app}}|}{\max\{1,f_{\text{dnn}}\}},
\]
and the relative approximation error with respect to the problem data size
\[
\textbf{apprnm}(\mathcal A):=\frac{|f_{\text{dnn}}-f_{\text{app}}|}{\max\{1,\|\mathcal A\|\}},
\]
to measure the approximation quality.
Note that due to the accuracy tolerance (the default is $10^{-6}$) set in solving the DNN relaxation
problem of \eqref{nonnegative-problem},
even if the matrix $M(\mathbf y^*)$ for the optimal $\mathbf y^*$ of \eqref{relaxation-primal} has rank one (thus the approximation is tight), we may still have $f_{\text{dnn}}\neq f_{\text{app}}$. But their difference should have the same magnitude as the
accuracy tolerance used.

Numerically, we regard the relaxation to be tight (e.g., when $\operatorname{rank}(M(\mathbf y^*))=1$) whenever the second largest singular value of $M(\mathbf y^*)$ is smaller than $1.0\times 10^{-6}$.

\begin{example}\label{exm-1}
This example comes from
\cite[Example~2]{De-De-Van:bes}. This is a tensor $\mathcal A$ in $\operatorname{Sym}(\otimes^3\mathbb R^2)$ with the independent entries being
\begin{align*}
&a_{111}= 1.5578, a_{2 2 2}= 1.1226, a_{112}= -2.4443, a_{221}= -1.0982.
\end{align*}
The relaxation is tight.
The best nonnegative rank-one approximation tensor found is
\[
\lambda =  1.5578 ,\ \mathbf x_*=( 1,  0)^\mathsf{T}.
\]
The errors $\textbf{apperr}(\mathcal A)=3.5924\times 10^{-6}$, and  $\textbf{apperrnm}(\mathcal A)=1.1142\times 10^{-6}$.
\end{example}

\begin{example}\label{exm-2}
This example comes from
 \cite[Example~3]{De-De-Van:bes}. This is a tensor $\mathcal A$ in $\otimes^4\mathbb R^2$ with nonzero entries being
\begin{align*}
&a_{1111}= 25.1, a_{1 21 2}= 25.6, a_{2121}= 24.8, a_{2222}= 23.
\end{align*}
This is a nonnegative and nonsymmetric tensor. The best nonnegative rank-one approximation tensor is the best rank-one approximation tensor (cf.\ \cite{Qi-Com-Lim:non}), which is found as
\[
\lambda = 25.6000,\ \mathbf x_*^1=\mathbf x_*^3=(1,0)^\mathsf{T},\ \mathbf x_*^2=\mathbf x_*^4=(0,1)^\mathsf{T}.
\]
The errors $\textbf{apperr}(\mathcal A)=9.1676\times 10^{-6}$, and  $\textbf{apperrnm}(\mathcal A)= 4.7616\times 10^{-6}$. The numerical computation is consistent\footnote{We remark on the different errors obtained in our computation and that in \cite{Nie-Wan:non}. Actually, Nie and Wang reported smaller approximation error. This is due to the facts that: (i) the SDP solvers are different (SDPNAL vs. SDPNAL+). Our formulation has an extra nonnegative constraint on the matrix variable. Although the SDPs have the same optimal values, the numerical computations adopt different termination accuracy tolerances. (ii) When we extract the solution for $\mathbf x_*$, we
also take the absolute values to make sure that $\mathbf x_*\geq 0$. This will introduce another difference. } with \cite[Example~3.11]{Nie-Wan:non}.

There is also a slight variation of $\mathcal{A}$, i.e., the tensor $\mathcal B$ in $\otimes^4\mathbb R^2$ with nonzero entries being (cf.\ \cite{De-De-Van:bes})
\begin{align*}
&b_{1111}= 25.1, b_{1 21 2}= 25.6, b_{2121}= 24.8, b_{2222}= 23, b_{1121} = 0.3, b_{2111}=0.3.
\end{align*}
The best nonnegative rank-one approximation tensor is
\[
\lambda = 25.6000,\ \mathbf x_*^1=\mathbf x_*^3=(1,0)^\mathsf{T},\ \mathbf x_*^2=\mathbf x_*^4=(0,1)^\mathsf{T},
\]
the same as that for $\mathcal A$.
The errors $\textbf{apperr}(\mathcal B)= 2.395\times 10^{-6}$, and  $\textbf{apperrnm}(\mathcal B)= 1.2439\times 10^{-6}$.
We see that $\textbf{apperr}(\mathcal B)$ is smaller than $\textbf{apperr}(\mathcal A)$. This is because there are more positive entries in $\mathcal B$ than those in $\mathcal A$, which improves the numerical computation and stability.
\end{example}

\begin{example}\label{exm-3}
This example comes from \cite[Example~2]{Qi:rat}. This is a tensor $\mathcal A$ in $\operatorname{Sym}(\otimes^3\mathbb R^3)$ with the independent entries being
\begin{align*}
&a_{111}= 0.0517, a_{ 1 1 2}= 0.3579, a_{1 1 3}= 0.5298, a_{ 1 2 2}= 0.7544, a_{ 1 2 3}= 0.2156,\\
& a_{1 3 3}= 0.3612, a_{2 2 2}= 0.3943, a_{2 2 3}= 0.0146, a_{2 3 3}= 0.6718, a_{3 3 3}= 0.9723.
\end{align*}
This is a also nonnegative tensor. The best nonnegative rank-one approximation tensor is the best rank-one approximation tensor, which is found as
\[
\lambda = 2.1110,\ \mathbf x_*=( 0.5204,  0.5113,  0.6839)^\mathsf{T}.
\]
The errors $\textbf{apperr}(\mathcal A)=3.261\times 10^{-6}$, and  $\textbf{apperrnm}(\mathcal A)=2.796\times 10^{-6}$. The numerical computation is consistent with \cite[Example~3.3]{Nie-Wan:non}.
\end{example}

\begin{example}\label{exm-4}
This example comes from \cite[Example~1]{Kof-Reg:bes}. It is a tensor $\mathcal A$ in $\operatorname{Sym}(\otimes^4\mathbb R^3)$ with the independent entries being
\begin{align*}
&a_{1111}= 0.2883, a_{ 11 1 2}= -0.0031, a_{11 1 3}= 0.1973, a_{ 11 2 2}= -0.2485, a_{ 11 2 3}=-0.2939,\\
& a_{11 3 3}=0.3847, a_{12 2 2}= 0.2972, a_{12 2 3}= 0.1862, a_{12 3 3}= 0.0919, a_{13 3 3}= -0.3619,\\
& a_{2222}= 0.1241, a_{2223}= -0.3420, a_{2233}= 0.2127, a_{2333}= 0.2727, a_{33 3 3}= -0.3054.
\end{align*}
The relaxation is not tight. The nonnegative rank-one approximation tensor found is
\[
\lambda = 0.6416,\ \mathbf x_*=(  0.9328,  0, 0.3603 )^\mathsf{T}.
\]
The errors $\textbf{apperr}(\mathcal A)= 5.8364\times 10^{-2}$, and  $\textbf{apperrnm}(\mathcal A)= 2.5910\times 10^{-2}$.
\end{example}

\begin{example}\label{exm-5}
This example comes from \cite[Example~3.6]{Kol-May:sft}. It is a tensor $\mathcal A$ in $\operatorname{Sym}(\otimes^3\mathbb R^3)$ with the independent entries being
\begin{align*}
&a_{111}= -0.1281, a_{ 11  2}= 0.0516, a_{11  3}=-0.0954, a_{ 1 2 2}= -0.1958, a_{ 12 3}=-0.1790,\\
& a_{1 3 3}=-0.2676, a_{2 2 2}= 0.3251, a_{2 2 3}= 0.2513, a_{2 3 3}= 0.1773, a_{3 3 3}= 0.0338.
\end{align*}
The relaxation is tight. The best nonnegative rank-one approximation tensor found is
\[
\lambda = 0.6187,\ \mathbf x_*=(  0,  0.8275,  0.5615 )^\mathsf{T}.
\]
The errors $\textbf{apperr}(\mathcal A)= 2.9194\times 10^{-6}$, and  $\textbf{apperrnm}(\mathcal A)=2.9194\times 10^{-6}$.
\end{example}

\begin{example}\label{exm-6}
This example comes from \cite[Example~3.8]{Nie-Wan:non}. It is a tensor $\mathcal A$ in $\operatorname{Sym}(\otimes^6\mathbb R^3)$ with the nonzero independent entries being
\begin{align*}
&a_{111111}= 2, a_{ 11112  2}= 1/3, a_{11 113 3}=2/5, a_{ 1 1222 2}= 1/3, a_{ 11223 3}=1/6,\\
& a_{1 1333 3}=2/5, a_{2 2 2222}= 2, a_{2 2223 3}= 2/5, a_{2 2333 3}= 2/5, a_{3 3 3333}=1.
\end{align*}
This is a nonnegative tensor, and the best nonnegative rank-one approximation tensor found is
\[
\lambda = 2,\ \mathbf x_*=(  0,  1,  0)^\mathsf{T}.
\]
The errors $\textbf{apperr}(\mathcal A)=2.2927\times 10^{-3}$, and  $\textbf{apperrnm}(\mathcal A)= 9.2979\times 10^{-4}$. The relaxation is not tight, but we know the
tensor found is the best nonnegative rank-one (cf.\ \cite[Example~3.3]{Nie-Wan:non}).
The numerical computation is consistent with \cite[Example~3.3]{Nie-Wan:non}.
\end{example}

\begin{example}\label{exm-7}
This example comes from \cite[Example~3.5]{Nie-Wan:non}. The tensor $\mathcal A\in\operatorname{Sym}(\otimes^m\mathbb R^n)$ with the entries being
\[
a_{i_1\dots i_m}=\sum_{j=1}^m\frac{(-1)^{i_j}}{i_j}
\]
The numerical computations are recorded in Table~\ref{table:exm-7}, in which \textbf{$\lambda$} is the norm of the computed best rank-one tensor, \textbf{Time} is the computation
 time taken in the format of \textit{hours:minutes:seconds}. In the table, the notation ``$8.9-6$" is a shorthand for ``$8.9\times 10^{-6}$" and so on so forth. We can see that in all cases, the method can find a very good best nonnegative rank-one approximation.

\begin{table}[h!]\caption{Computational results for Example~\ref{exm-7}
}\label{table:exm-7}
\centering
\begin{tabular}{|c | c | r | r | c | c| }
\hline
m & n &\multicolumn{1}{c|}{Time} & \multicolumn{1}{c|}{$\lambda$} & \textbf{apperr} & \textbf{apperrnm}\\ \hline
3 & 10 & 4.3 & 9.4878 & 8.9-6& 3.8-6\\ \hline
 3 & 20 & 29 & 19.2494 & 5.9-10& 2.5-10\\ \hline
 3 & 30 & 5:38 & 28.7246 & 3.0-5& 1.3-5\\ \hline
3&50&2:13:01&47.167&4.5-5&1.9-5\\  \thickhline
 4 & 10 & 1.4 & 33.4925 & 1.0-5& 4.4-6\\ \hline
 4 & 20 & 35 & 97.6098 & 2.5-8& 1.0-8\\ \hline
4 & 30 & 4:52 & 179.5584 & 5.7-9& 2.4-9\\ \hline
4&50&32:49&382.44&1.3-9&5.7-10\\
\thickhline
5 &  5 & 1 & 20.8284 & 1.1-5& 3.1-6\\ \hline
5 & 10 & 25 & 114.8631 & 1.5-8& 6.1-9\\ \hline
5&20&4:27:04&480.19&1.2-5&5.1-6
\\ \thickhline
6 &  5 & 0.59 & 46.6667 & 1.4-5& 3.4-6\\ \hline
6 & 10 & 13 & 386.0448 & 3.2-8& 1.2-8\\ \hline
6&20&2:32:18&2319.3&3.4-5&1.3-5 \\  \thickhline
7&5&6.4&103.02&7.8-6&1.6-6\\ \hline
7&10&11:29&1278.4&3.5-8&1.2-8\\  \thickhline
8&5&5.8&225.37&2.7-5&5.1-6\\ \hline
8&10&3:41&4186.1&1.2-7&4.4-8\\ \hline
\end{tabular}
\end{table}
\end{example}

\begin{example}\label{exm-8}
This example comes from \cite[Example~3.6]{Nie-Wan:non}. The tensor $\mathcal A\in\operatorname{Sym}(\otimes^m\mathbb R^n)$ with the entries being
\[
a_{i_1\dots i_m}=\sum_{j=1}^m\operatorname{arctan}\bigg(\frac{(-1)^{i_j}i_j}{n}\bigg).
\]
The numerical computations are recorded in Table~\ref{table:exm-8}.
\begin{table}[h!]\caption{Computational results for Example~\ref{exm-8}}\label{table:exm-8}
\centering
\begin{tabular}{|c | c | r | r | c | c |}
\hline
m & n
&\multicolumn{1}{c|}{Time} & \multicolumn{1}{c|}{$\lambda$}
& \textbf{apperr} & \textbf{apperrnm}\\ \hline
 3 & 10 & 6.6 & 21.1979 & 7.5-6& 5.5-6\\ \hline
 3 & 20 & 1:24 & 55.5867 & 5.5-6& 3.9-6\\ \hline
3&50&2:26:47&209.19&1.1-5&7.9-6\\
 \thickhline
 4 & 10 & 2.4  & 77.0689 & 7.5-6& 5.4-6\\ \hline
 4 & 20 & 44  & 282.9708 & 7.0-6& 4.8-6\\ \hline
4&50&2:56:48&1672.7&3.7-6&2.5-6\\ \thickhline
5 &  5 & 17  & 30.5470 & 1.2-3& 5.1-4\\ \hline
5 & 10 & 39  & 273.3958 & 2.7-5& 1.9-5\\ \hline
5&20&5:20:51&1407.8&1.7-9&1.1-9\\ \thickhline
6&10&21 &953.06&4.2-6&3.1-6\\ \hline
6&20&2:56:40&6890.5&1.3-9&9.0-10\\ \thickhline
7&5&44 &162.21&2.8-3&1.0-3\\ \hline
7&10&28:44 &3280&7.6-5&5.6-5\\ \thickhline
8&5&1.8&370.33&6.0-6&2.0-6\\ \hline
8&10&14:00&11178&4.4-5&3.2-5\\ \hline
\end{tabular}
\end{table}
\end{example}

\begin{example}\label{exm-9}
This example comes from \cite[Example~3.7]{Nie-Wan:non}. The tensor $\mathcal A\in\operatorname{Sym}(\otimes^m\mathbb R^n)$ with the entries being
\[
a_{i_1\dots i_m}=\sum_{j=1}^m (-1)^{i_j}\log(i_j).
\]
The numerical computations are recorded in Table~\ref{table:exm-9}.
\begin{table}[h!]\caption{Computational results for Example~\ref{exm-9}}\label{table:exm-9}
\centering
\begin{tabular}{|c | c | r | r | c | c |}
\hline
m & n
&\multicolumn{1}{c|}{Time} & \multicolumn{1}{c|}{$\lambda$}
& \textbf{apperr} & \textbf{apperrnm}\\ \hline
 3 & 10  &  21 & 68.0631 & 3.7-6& 2.7-6\\ \hline
  3 & 20  &  !:02 & 246.1904 & 9.7-6& 6.8-6\\ \hline
3&50&2:25:17&1289.8&1.3-5&8.8-6
\\ \thickhline
 4 & 10  &  2.3 & 248.2981 & 8.5-6& 6.2-6\\ \hline
 4 & 20  &  36 & 1253.3842 & 1.0-5& 7.2-6\\ \hline
4&50&2:12:51&10306&3.7-5&2.4-5\\
 \thickhline
 5 &  5  &  18& 69.8570 & 2.8-4& 1.3-4\\ \hline
5 & 10  &  1:10 & 883.2849 & 4.7-10& 3.4-10\\ \hline
5&20&5:00:10&6236.7&5.9-10&4.1-10\\ \thickhline
6&10&21&3086.6&2.1-5&1.6-5\\ \hline
6&20&3:43:30&30529&2.6-5&1.8-5\\ \thickhline
7&5&42&383.84&7.0-4&3.1-4\\ \hline
7&10&2:09:53&10645&1.1-4&8.6-5\\ \thickhline
8&5&2.1&889.93&7.0-6&3.0-6\\ \hline
8&10&14:09&36349&2.5-5&1.8-5\\ \hline
\end{tabular}
\end{table}
\end{example}

\begin{example}\label{exm-10}
This example comes from \cite[Example~3.10]{Nie-Wan:non}. The tensor $\mathcal A\in\operatorname{Sym}(\otimes^m\mathbb R^n)$ with the entries being
\[
a_{i_1\dots i_m}=\sin\big(\sum_{j=1}^mi_j\big).
\]
The numerical computations are recorded in Table~\ref{table:exm-10}.
\begin{table}[h!]\caption{Computational results for Example~\ref{exm-10}}\label{table:exm-10}
\centering
\begin{tabular}{|c | c | r | r | c | c |}
\hline
m & n
&\multicolumn{1}{c|}{Time} & \multicolumn{1}{c|}{$\lambda$}
& \textbf{apperr} & \textbf{apperrnm}\\ \hline
3 & 10  &  2.2 & 2.9121 & 5.2-1& 1.4-1\\ \hline
 3 & 15  &  4.8 & 5.6006 & 4.8-1& 1.3-1\\ \hline
3 & 20  &  28 & 10.2134 & 3.8-1& 9.9-2\\ \hline
3&50&57:24&44.541&2.9-1&7.3-2\\
 \thickhline
 4 & 10  &  1.5 & 8.0140 & 4.1-1& 7.8-2\\ \hline
4 & 15  &  4.3 & 22.0576 & 2.8-1& 5.4-2\\ \hline
4 & 20  &  14 & 21.9602 & 5.6-1& 1.0-1\\ \hline
4&50&48:15&158.22&4.9-1&8.9-2\\
\thickhline
5 &  5  &  0.68 & 2.7205 & 6.6-1& 1.3-1\\ \hline
5 & 12  &  1:36 & 17.3683 & 6.4-1& 9.1-2\\ \thickhline
6&10&24&22.85&6.6-1&6.4-2\\ \hline
6&20&2:18:31&103.43&8.2-1&8.3-2\\ \thickhline
\end{tabular}
\end{table}
\end{example}

\begin{example}\label{exm-11}
This example comes from \cite[Example~3.14]{Nie-Wan:non}. The tensor $\mathcal A\in\otimes^m\mathbb R^n$ with the entries being
\[
a_{i_1\dots i_m}=\cos\big(\sum_{j=1}^mj\cdot i_j\big).
\]
The numerical computations are recorded in Table~\ref{table:exm-11}.
\begin{table}[h!]\caption{Computational results for Example~\ref{exm-11}}\label{table:exm-11}
\centering
\begin{tabular}{|c | c | r | r | c | c| }
\hline
m & n
&\multicolumn{1}{c|}{Time} & \multicolumn{1}{c|}{$\lambda$}
& \textbf{apperr} & \textbf{apperrnm}\\ \hline
3 &  4  &  13 & 2.4412 & 3.7-3& 1.6-3\\ \hline
3 &  5  &  18 & 2.9581 & 4.3-2& 1.6-2\\ \hline
3 &  6  &  32 & 2.8464 & 2.5-1& 9.1-2\\ \thickhline
4 &  2  &  0.88 & 1.2392 & 9.4-6& 4.2-6\\ \hline
4 &  3  &  95 & 1.7608 & 9.7-2& 2.9-2\\ \thickhline
5 &  2  &  8.5 & 1.4061 & 1.0-5& 3.8-6\\ \hline
\end{tabular}
\end{table}
\end{example}

\begin{example}\label{exm-12}
This example comes from \cite[Example~3.16]{Nie-Wan:non}. The tensor $\mathcal A\in\otimes^m\mathbb R^n$ with the entries being
\[
a_{i_1\dots i_m}=\sum_{j=1}^m(-1)^{j+1}\cdot j\cdot \exp(-i_j).
\]
The numerical computations are recorded in Table~\ref{table:exm-12}.
\begin{table}[h!]\caption{Computational results for Example~\ref{exm-12}}\label{table:exm-12}
\centering
\begin{tabular}{|c | c | r | r | c | c| }
\hline
m & n
&\multicolumn{1}{c|}{Time} & \multicolumn{1}{c|}{$\lambda$}
& \textbf{apperr} & \textbf{apperrnm}\\ \hline
3 &  4  &  16 & 636.9974 & 3.3-6& 3.0-6\\ \hline
 3 &  5  &  1:08 & 2230.7114 & 3.7-6& 3.3-6\\ \hline
 3 &  6  &  4:24 & 7411.5508 & 4.7-6& 4.2-6\\
\thickhline
 4 &  2  &  1 & 16.0454 & 3.4-6& 8.5-7\\ \hline
 4 &  3  &  38 & 148.8945 & 2.5-6& 9.3-7\\ \thickhline
5 &  2  &  15 & 123.1144 & 5.3-6& 5.0-6\\ \hline
\end{tabular}
\end{table}
\end{example}

\begin{example}\label{exm-13}
This example comes from \cite[Example~3.18]{Nie-Wan:non}. The tensor $\mathcal A\in\otimes^m\mathbb R^n$ with the entries being
\[
a_{i_1\dots i_m}=\tan\bigg(\sum_{j=1}^m(-1)^{j+1}\cdot \frac{i_j}{j}\bigg).
\]
The numerical computations are recorded in Table~\ref{table:exm-13}.
\begin{table}[h!]\caption{Computational results for Example~\ref{exm-13}}\label{table:exm-13}
\centering
\begin{tabular}{|c | c | r | r | c | c| }
\hline
m & n
&\multicolumn{1}{c|}{Time} & \multicolumn{1}{c|}{$\lambda$}
& \textbf{apperr} & \textbf{apperrnm}\\ \hline
3 &  4  &  9.2 & 15.3005 & 3.3-6& 1.7-6\\ \hline
3 &  5  &  1:09 & 22.1109 & 1.2-1& 5.9-2\\ \hline
 3 &  6  &  2:52 & 22.0969 & 2.0-1& 8.2-2\\
\hline
 4 &  2  &  1.0 & 7.1928 & 1.0-5& 9.3-7\\ \hline
4 &  3  &  8.8 & 14.8249 & 1.4-4& 1.5-5\\ \thickhline
5 &  2  &  19 & 242.2146 & 1.6-6& 1.6-6\\ \hline
\end{tabular}
\end{table}
\end{example}

\begin{example}[Determinant Tensor/Levi-Civita Tensor]\label{exm-14}
This example considers the tensor $\mathcal A\in\otimes^n\mathbb R^n$ of the determinant for $n\times n$ matrices. Given a matrix $A=[a_{ij}]\in\mathbb R^{n\times n}$, its determinant is
\[
\operatorname{det}_n(A):=\sum_{\sigma\in\mathfrak G(n)}\text{sign}(\sigma)\prod_{i=1}^na_{i\sigma(i)}
\]
where $\mathfrak G(n)$
is  the permutation group
on $n$ elements and $\text{sign}(\sigma)$ is the sign of a permutation $\sigma$.
We can view
$\operatorname{det}_n(A)$ as a multilinear form over the groups of variables $\{a_{11},\dots,a_{1n}\}$, $\dots$, $\{a_{n1},\dots,a_{nn}\}$. Likewise, it can be uniquely regarded as a tensor in $\otimes^n\mathbb R^n$.
When $n=3$, the determinant tensor $\mathcal A\in \otimes^3\mathbb R^3 $ has the nonzero entries
\begin{align*}
&a_{123}=a_{231}=a_{312}=-a_{132}=-a_{213}=-a_{321}=1.
\end{align*}
This tensor is usually referred as the \textit{Levi-Civita symbol} \cite[Section~15.2]{Lim:hyp}. It captures the structure constants of the Lie algebra $\mathfrak{so}(3)$. In general, we call the tensor
$\mathcal A\in\otimes^n\mathbb R^n$ with the nonzero entries being
\[
a_{i_1\dots i_n}=\operatorname{sign}(\sigma)\ \text{if }i_j=\sigma(j)\ \text{for all }j=1,\dots,n
\]
as the \textit{Levi-Civita tensor} of order $n$. It is easy to see that the Levi-Civita tensor is the determinant tensor.

When $n=5$, the matrix dimension of the DNN relaxation problem is \textbf{$117649\times  117649$}, which is far too large for our computer's memory. Thus, we compute the problems up to $n=4$.
The approximations are all tight, and the computed best rank-one approximation tensor is
\[
\mathbf e_1\otimes\dots\otimes\mathbf e_n
\]
with $\mathbf e_i\in\mathbb R^n$ the $i$th standard basis vector.

The numerical computations are recorded in Table~\ref{table:exm-14}.
\begin{table}[h!]\caption{Computational results for Example~\ref{exm-14}}\label{table:exm-14}
\centering
\begin{tabular}{ |c | r | r | c | c |}
\hline
 n
 &\multicolumn{1}{c|}{Time} & \multicolumn{1}{c|}{$\lambda$}
 & \textbf{apperr} & \textbf{apperrnm}\\ \hline
 2 &   0.25  & 1.0000 & 1.6-6& 1.1-6\\ \hline
  3 &   0.47  & 1.0000 & 2.9-6& 1.2-6\\ \hline
 4 &   1:17 & 1.0000 & 9.5-6& 1.9-6\\ \hline
\end{tabular}
\end{table}
\end{example}

\begin{example}[Permanent Tensor]\label{exm-15}
This example considers the tensor $\mathcal A\in\otimes^n\mathbb R^n$ of the permanent for $n\times n$ matrices. Given a matrix $A=[a_{ij}]\in\mathbb R^{n\times n}$, its permanent is
\[
\operatorname{pet}_n(A):=\sum_{\sigma\in\mathfrak G(n)} \prod_{i=1}^na_{i\sigma(i)},
\]
$\operatorname{pet}_n(A)$ can be viewed as a multilinear form over the groups of variables $\{a_{11},\dots,a_{1n}\}$, $\dots$, $\{a_{n1},\dots,a_{nn}\}$. Likewise, it can be uniquely regarded as a tensor in $\otimes^n\mathbb R^n$. This tensor is nonnegative. It follows from \cite{Der:nuc} that
the best nonnegative rank-one approximation tensor has norm
\[
\lambda_* = \frac{n!}{n^{n/2}}.
\]
Similar to the determinant tensor, we can only handle the permanent tensor up to $n=4$. The relaxation is tight only for $n=2$. The computed best rank-one approximation tensor is
\[
\mathbf e_{\sigma(1)}\otimes\dots\otimes\mathbf e_{\sigma(n)}
\]
for all permutations $\sigma\in\mathfrak G(n)$, since the tensor is nonnegative.

The numerical computations are recorded in Table~\ref{table:exm-15}.
\begin{table}[h!]\caption{Computational results for Example~\ref{exm-15}}\label{table:exm-15}
\centering
\begin{tabular}{ |c | r | c | c | c | c| }
\hline
 n & \multicolumn{1}{c|}{Time}
  & $\frac{n!}{n^{n/2}}$
 &$\lambda$ & \textbf{apperr} & \textbf{apperrnm}\\ \hline
  2 &   0.21  & 1&1.0000 & 1.2-6& 8.8-7\\ \hline
  3 &   0.67  & 1.1547&1.0000 & 1.3-1& 6.3-2\\ \hline
 4 &   2:06&1.5 & 1.0000 & 3.3-1& 1.0-1\\ \hline
\end{tabular}
\end{table}
\end{example}

\begin{example}[Matrix Multiplication Tensor]\label{exm-16}
Given positive integers $m,n,q$, this example considers the tensor $\mathcal A\in\mathbb R^{mq}\otimes\mathbb R^{mn}\otimes\mathbb R^{nq}$
arising from the matrix multiplication of matrices of sizes $m\times n$ and $n\times q$. Given two matrices $A=[a_{ij}]\in\mathbb R^{m\times n}$ and $B=[b_{jk}]\in\mathbb R^{n\times q}$, their matrix multiplication $C=AB\in\mathbb R^{m\times q}$ is
\[
[c_{ik}]=\big[\sum_{j=1}^n a_{ij}b_{jk}\big].
\]
The bilinear map $\mathbb R^{mn}\times\mathbb R^{nq}\rightarrow \mathbb R^{mq}$ can be represented as a tensor $\mathcal A\in\mathbb R^{mq}\otimes\mathbb R^{mn}\otimes\mathbb R^{nq}$, such that
\[
C = \langle\mathcal A,A\otimes B\rangle_{2,3:1,2},
\]
where the equality is of course understood in the sense of the standard isomorphism and $\langle \cdot,\cdot\rangle_{2,3:1,2}$ is a tensor contraction by contracting the second and third indices of the first argument with the first and second indices of the second argument (cf.\ \cite{Lim:hyp}). This tensor is nonnegative.

The DNN relaxations are always tight. The computed best nonnegative rank-one tensor is
\[
\mathbf e_i\otimes\mathbf e_j\otimes\mathbf e_k
\]
for some standard basis vectors $\mathbf e_i\in\mathbb R^{mn}$, $\mathbf e_j\in\mathbb R^{nq}$, and $\mathbf e_k\in\mathbb R^{mq}$.

The numerical computations are recorded in Table~\ref{table:exm-16}.
The matrix multiplication tensor $\mathcal A$ is not invariant with respect to $m$, $n$ and $q$. However, we can
see from the table that when $\{m,n,q\}$ is a fixed set, the tensors share the same approximation errors. 
\begin{table}[h!]\caption{Computational results for Example~\ref{exm-16}}
\label{table:exm-16}
\centering
\begin{tabular}{| c | c | c | r | r | c | c |}
\hline
 m& n & q
 &\multicolumn{1}{c|}{Time} & \multicolumn{1}{c|}{$\lambda$}
 & \textbf{apperr} & \textbf{apperrnm}\\ \hline
  2 &  2 & 2 &  1.7 & 1.0000 & 3.6-6& 1.2-6\\ \thickhline
2 &  2 & 3 &  7.1 & 1.0000 & 2.9-6& 8.5-7\\
\hline
 3 &  2 & 2 &  6.9 & 1.0000 & 2.9-6& 8.5-7\\ \hline
 2 &  3 & 2 &  7.0 & 1.0000 & 2.9-6& 8.5-7\\ \thickhline
 2 &  3 & 3 &  43& 1.0000 & 3.5-6& 8.2-7 \\ \hline
 3 &  2 & 3 &  44 & 1.0000 & 3.5-6& 8.2-7\\ \hline
 3 &  3 & 2 &  34 & 1.0000 & 3.5-6& 8.2-7\\ \thickhline
 2 &  2 & 4 &  21 & 1.0000 & 2.3-6& 5.8-7\\ \hline
2 &  2 & 5 &  59 & 1.0000 & 2.0-6& 4.6-7\\ \hline
 2 &  2 & 6 &  2:14 & 1.0000 & 6.4-6& 1.3-6\\ \hline
2 &  2 & 7 &  5:52 & 1.0000 & 4.3-7& 8.1-8\\ \hline
2 &  2 & 8 &  11:43 & 1.0000 & 8.3-6& 1.4-6\\ \hline
 2 &  2 & 9 &  20:59 & 1.0000 & 2.3-6& 3.9-7\\ \thickhline
 2 &  3 & 4 &  2:25 & 1.0000 & 7.2-7& 1.4-7\\ \hline
 2 &  3 & 5 &  6:23 & 1.0000 & 1.7-6& 3.2-7\\
\hline
2 &  3 & 6 &  26:37 & 1.0000 & 2.6-6& 4.4-7\\
\thickhline
3 &  3 & 3 &  3:56 & 1.0000 & 4.5-6& 8.7-7\\ \hline
3 &  3 & 4 &  21:26 & 1.0000 & 2.4-6& 4.0-7\\
\hline
2 &  4 & 4 &  10:05 & 1.0000 & 2.5-6& 4.4-7\\ \hline
2 &  4 & 5 &  39:56& 1.0000 & 1.0-6& 1.6-7\\ \hline
\end{tabular}
\end{table}

\end{example}


\subsection{Copositivitiy of tensors}
In this section, we test some tensors for their copositivities.
Let $f_{\text{dnn}}$ be the optimal value of the DNN relaxation and $f_{\text{app}}$ be the approximation value found as before. Then
\begin{enumerate}
\item if $f_{\text{dnn}}\geq 0$, then we can conclude that the tensor is copositive,
\item if $f_{\text{app}}<0$, then we can conclude that the tensor is not copositive.
\end{enumerate}

\begin{example}\label{exm-18}
This example comes from \cite[Page~237]{Qi:cop}. It is a tensor $\mathcal A$ in $ \operatorname{Sym}(\otimes^3\mathbb R^3)$ with nonzero entries being
\begin{align*}
&a_{113}= 2, a_{223}=2, a_{123}=-1.
\end{align*}
It can be show that
\[
\langle\mathcal A,\mathbf x^{\otimes 3}\rangle=6x_3(x_1^2+x_2^2-x_1x_2),
\]
and hence $\mathcal A$ is copositive. We have that
\[
f_{\text{dnn}}=  9.3650\times 10^{-15},\ \text{and }f_{\text{app}}= 2.3094.
\]
Therefore we can conclude that the numerical computation gives the correct answer.
\end{example}

\begin{example}\label{exm-19}
This example comes from \cite[Theorem~10]{Qi:cop}. It is a tensor $\mathcal A \in \operatorname{Sym}(\otimes^m\mathbb R^n)$ such that
\[
a_{ii\dots i}\geq -\sum\{a_{ii_2\dots i_m} : (i,i_2,\dots,i_m)\neq (i,i,\dots,i)\ \text{and }a_{ii_2\dots i_m}<0\}\ \text{for all }i=1,\dots,n.
\]
Tensors satisfying the above assumption are always copositive. For each case, we randomly generate the tensor $\mathcal A \in \operatorname{Sym}(\otimes^m\mathbb R^n)$ and set
\[
a_{ii\dots i}=10^{-6}-\sum\{a_{ii_2\dots i_m} : (i,i_2,\dots,i_m)\neq (i,i,\dots,i)\ \text{and }a_{ii_2\dots i_m}<0\}\ \text{for all }i=1,\dots,n.
\]
We simulate \textbf{rep} times for each case, and use \textbf{prob} to denote the percentage
of the instances which are tested as copositive. From the theory, we know that \textbf{prob} should be one.
The numerical computations are recorded in Table~\ref{table:exm-19}.
\begin{table}[h!]\caption{Computational results for Example~\ref{exm-19}}
\label{table:exm-19}
\centering
\begin{tabular}{|c | c | c | c | c |  c| }
\hline
m & n & \textbf{rep} & Time (min;mean;max) & $f_{\text{dnn}}$ (min;mean;max) & \textbf{prob}\\ \hline
3 &  2 &100 &  0.11 ; 0.24;  1.1 & 0.0074;  0.5658 ; 1.4161 & 1.0000\\ \hline
3 &  4 & 100 &  0.20 ; 0.81;  13 & 0.9781;  1.9254 ; 3.2333 & 1.0000\\ \hline
4 &  4 & 100 &  0.21 ; 1.3;  41 & 1.7688;  3.4358 ; 5.5922 & 1.0000\\ \hline
4 & 10 &  20 &  2.4 ; 6.9;  13 & 30.8177;  35.9170 ; 41.1639 & 1.0000\\ \hline
\end{tabular}
\end{table}
\end{example}

\begin{example}[Random Examples]\label{exm-20}
We test randomly generated tensors to estimate the probability for the tensors to be copositive.
Each entry of the tensor is generated randomly uniformly from $[-1,1]$.
The numerical computations are recorded in Table~\ref{table:exm-20}.

\begin{table}[h!]\caption{Computational results for Example~\ref{exm-20}}\label{table:exm-20}
\centering
\begin{tabular}{|c | c | c | c | c |  c| }\hline
m & n & \textbf{rep} & Time (min;mean;max) & $f_{\text{dnn}}$ (min;mean;max) & \textbf{prob}\\ \hline
 3 &  2 & 100000 &  0.04 ; 0.22;  3:01 & -1.4627;  0.3318 ; 2.4426 & 0.7878\\ \hline
3 &  3 & 100 &  0.05; 0.41;  2.3 & -1.2454;  0.6822 ; 2.7819 & 0.7400\\ \hline
 3 &  4 & 100 &  0.21; 2.2;  2:28& -1.7732;  1.1233 ; 2.9112 & 0.8300\\ \thickhline
4 &  2 & 100 &  0.07; 0.21;  1.7& -0.7297;  0.7144 ; 2.4577 & 0.8800\\ \hline
 4 &  3 & 100 &  0.10; 0.37;  4.8 & -0.7667;  1.2694 ; 4.3429 & 0.9000\\ \hline
4 &  4 & 100 &  0.14; 0.52;  4.2 & -0.7705;  2.0959 ; 5.1505 & 0.9500\\ \thickhline
 5 &  2 & 100 &  0.06; 0.65;  15 & -1.3392;  0.7664 ; 3.9843 & 0.7900\\ \hline
 5 &  3 & 100 &  0.18; 3.7;  4:36 & -2.1679;  1.0422 ; 5.5006 & 0.6100\\ \hline
\end{tabular}
\end{table}
\end{example}

\subsection{Comparison with SDPNAL}\label{sec:comparison}
We have already mentioned before that the DNN problems can be formulated as standard SDP problems by introducing extra variables and constraints, and the
resulting standard SDP problems can also be solved directly by
the solver SDPNAL \cite{ZST-10}.
But we also know that
the new enhanced version SDPNAL+ is designed with the focus on solving DNN problems \cite{YST-15}. Thus, we would expect the latter solver to be
more efficient in solving our DNN problems \eqref{relaxation-dnn}.
To see this, finally, we have a comparison between the performance of SDPNAL and SDPNAL+ in solving Example \ref{exm-11}.
The results are recorded in Tables~\ref{table:exm-11-sdpnal} and \ref{table:exm-11-sdpnal+} respectively. We can see the superiority of SDPNAL+ in solving
large scale problems coming from the cases $m\in\{5,6,7\}$.

\begin{table}[h!]
\caption{SDPNAL on Example~\ref{exm-11}}\label{table:exm-11-sdpnal}
\centering
\begin{tabular}{| c | c | c | r | r | c | c |}  \hline
m & n &(\# sdp; \# con.)
 &\multicolumn{1}{c|}{Time} & \multicolumn{1}{c|}{$\lambda$}
& \textbf{apperr} & \textbf{apperrnm}\\ \hline
3&8&729; 441046&11:59&5.0068&0.12&0.04\\ \hline
3&10&1331; 1485397&2:27:36&6.9074&0.15&0.05\\ \hline
3&12&2197; 4075436&5:48:46&7.2059&0.27&0.09\\ \hline
4&4&625; 340626&3:52&3.1353&0.14&0.04\\ \hline
4&5&1296; 1486432&1:00:53&3.8602&0.33&0.10\\ \hline
4&6&2401; 5152547&10:33:23&3.8465&0.45&0.12\\ \hline
5&3&1024; 949601&1:32:58&1.1975&0.52&0.12\\ \hline
6&2&729; 485515&9:52&1.9076&7.5-8&2.5-8\\ \hline
7&2&2187; 4505221&5:56:23&2.5485&0.10&0.03\\   \hline
\end{tabular}
\end{table}
\begin{table}[h!]
\caption{SDPNAL+ on Example~\ref{exm-11}}\label{table:exm-11-sdpnal+}
\centering
\begin{tabular}{| c | c | c | r | r | c | c |} \hline
m & n &( \# sdp; \# con.)
&\multicolumn{1}{c|}{Time} & \multicolumn{1}{c|}{$\lambda$}
& \textbf{apperr} & \textbf{apperrnm}\\ \hline
3&8&729; 174961&15:04&5.0069&0.12&0.04\\ \hline
3&10&1331; 598951&1:52:06&6.9068&0.07&0.02\\ \hline
3&12&2197; 1660933&8:49:04&7.2052&0.26&0.09\\ \hline
4&4&625; 145001&3:32&3.1353&0.14&0.04\\ \hline
4&5&1296; 645976&1:13:09&3.8602&0.31&0.10\\ \hline
4&6&2401; 2268946&11:27:18&3.8465&0.45&0.12\\ \hline
5&3&1024; 424801&\textbf{36:00}&1.1974&0.50&0.11\\ \hline
6&2&729; 219430&\textbf{2:04}&1.9076&3.0-5&1.0-5\\ \hline
7&2&2187; 2112643&\textbf{55:40}&2.5485&0.10&0.03\\ \hline
\end{tabular}
\end{table}

\section{Conclusions}\label{sec:con}

This article studied the problem of minimizing a multi-form over the nonnegative multi-sphere. This problem is a special polynomial optimization problem. Although standard SOS relaxation method can be employed to solve this problem, there are
computational advantages to consider the more specialized approach in this paper. Take the biquadratic case for example, i.e., $d_1=d_2=2$. The matrix in the resulting SDP is of dimension 
$\frac{(n_1+n_2+1)(n_1+n_2)}{2}$. However, the matrix dimension in the DNN relaxation method introduced here is $n_1n_2$.
We can see that the latter method provides a linear matrix conic optimization problem with  matrix size that is about half of the former when $n_1=n_2$, and the ratio is even smaller when
$n_1 \ll n_2$ or $n_2\ll n_1$. Given the current limitations of SDP solvers for handling large scale problems with high-dimensional matrix variables (cf.\ \cite{ZST-10,SDPA,SDM,SDPT3}), the DNN method proposed in this article is promising.
Our approach is made practical by the
recent solver SDPNAL+ \cite{YST-15}, which is designed to
efficiently handle large-scale problems with a particular focus on DNN problems
having moderate matrix dimensions while allowing the
number of linear constraints to be very large.

The method is applied to the problem of finding the best nonnegative rank-one approximation of a given tensor and the problem of testing the
copositivity of a given tensor. Based on the promising numerical results, we
are motivated to carry out further investigations of the DNN relaxation methods for multi-form optimization over the nonnegative multi-sphere in the future.


\section*{Acknowledgement}
Most of the work was carried out during the first author's postdoctoral research in the Department of Mathematics at National University of Singapore.
The research of Shenglong Hu was supported in part by National Science Foundation of China (Grant No. 11771328), Young Elite Scientists Sponsorship Program by Tianjin, and Innovation Research Foundation of Tianjin University
(Grant Nos. 2017XZC-0084 and 2017XRG-0015). The research of Defeng Sun was supported in part by a start-up research grant from the Hong Kong Polytechnic University. The research of Kim-Chuan Toh was supported in part by the Ministry of Education, Singapore, Academic Research Fund under Grant R-146-000-257-112.
\bibliographystyle{model6-names}

\end{document}